\documentclass{amsart}%
\usepackage{palatino, mathpazo}
\usepackage{amsfonts}
\usepackage{amsmath}
\usepackage{amssymb,latexsym,xcolor}
\usepackage{graphicx}
\usepackage[mathscr]{eucal}
\usepackage{harpoon}
\usepackage{amssymb}
\usepackage[
linktocpage=true,colorlinks,citecolor=magenta,linkcolor=blue,urlcolor=magenta]{hyperref}
\allowdisplaybreaks[4]
\providecommand{\U}[1]{\protect \rule{.1in}{.1in}}

\newtheorem{theorem}{Theorem}[section]

\newtheorem{corollary}[theorem]{Corollary}
\newtheorem{definition}[theorem]{Definition}
\newtheorem{example}[theorem]{Example}
\newtheorem{Lemma}[theorem]{Lemma}
\newtheorem{Proposition}[theorem]{Proposition}
\newtheorem{Theorem}{Theorem}

\theoremstyle{remark}
\newtheorem{remark}[theorem]{Remark}

\numberwithin{equation}{section}

\def\R{\mathbb R}

\def\dl{\Delta}
\def\e{\varepsilon}
\def\om{\Omega}
\def\iy{\infty}

\def\rt{\rightarrow}

\def\bt{\begin{theorem}}
\def\et{\end{theorem}}
\def\bl{\begin{Lemma}}
\def\el{\end{Lemma}}
\def\bd{\begin{definition}}
\def\ed{\end{definition}}
\def\bc{\begin{corollary}}
\def\ec{\end{corollary}}
\def\bprop{\begin{Proposition}}
\def\eprop{\end{Proposition}}
\def\bp{\begin{proof}}
\def\ep{\end{proof}}
\def\bx{\begin{example}}
\def\ex{\end{example}}
\def\br{\begin{remark}}
\def\er{\end{remark}}
\def\be{\begin{equation}}
\def\ee{\end{equation}}
\def\bal{\begin{align}}
\def\bn{\begin{enumerate}}
\def\en{\end{enumerate}}
\def\eal{\end{align}}

\def\bg{\begin{align*}}
\def\eg{\end{align*}}
\def\bcs{\begin{cases}}
\def\ecs{\end{cases}}

\def\RNum#1{\uppercase\expandafter{\romannumeral #1\relax}}

\def\bean{\begin{eqnarray*}}
\def\eean{\end{eqnarray*}}

\begin{document}
\title[Lane-Emden Problems]{Sharp estimates and non-degeneracy of low energy nodal solutions for the Lane-Emden equation in dimension two}
\author{Zhijie Chen}
\address{Department of Mathematical Sciences, Yau Mathematical Sciences Center,
Tsinghua University, Beijing, 100084, China}
\email{zjchen2016@tsinghua.edu.cn}
\author{Zetao Cheng}
\address{Department of Mathematical Sciences,
Tsinghua University, Beijing, 100084, China}
\email{chengzt20@mails.tsinghua.edu.cn}
\author{Hanqing Zhao}
\address{Department of Mathematical Sciences,
Tsinghua University, Beijing, 100084, China}
\email{zhq20@mails.tsinghua.edu.cn}



\begin{abstract}
We study the Lane-Emden problem 
\[
	\begin{cases}
    -\Delta u_p =|u_p|^{p-1}u_p\quad&\text{in}\quad \Omega,\\
    u_p=0 \quad&\text{on}\quad\partial\Omega,
    \end{cases}
\]
where $\Omega\subset\mathbb R^2$ is a smooth bounded domain and $p>1$ is sufficiently large. We obtain sharp estimates and non-degeneracy of low energy nodal solutions $u_p$ (i.e. nodal solutions satisfying 
$\lim_{p\to+\infty}p\int_{\Omega}|u_p|^{p+1}dx=16\pi e$).
As applications,
we prove that the comparable condition $p(\|u_p^+\|_{\infty}-\|u_p^-\|_{\infty})=O(1)$ holds automatically for least energy nodal solutions, which confirms a conjecture raised by (Grossi-Grumiau-Pacella, Ann. I. H. Poincar\'{e}-AN, 30 (2013), 121-140) and (Grossi-Grumiau-Pacella,  J. Math. Pures Appl. 101 (2014), 735–754). 
\end{abstract}
\maketitle

\section{Introduction}
\label{section-1}

In this paper, we study the Lane–Emden equation
\begin{align}\label{p_p}
     \begin{cases}\tag{$\mathscr{P}_p$}
    -\Delta u_p =|u_p|^{p-1}u_p\quad&\text{in}\quad \Omega,\\
    u_p=0 \quad&\text{on}\quad\partial\Omega,
    \end{cases}
\end{align}
where $\Omega\subset\mathbb R^2$ is a smooth bounded domain and $p>1$.
This equation has been widely studied in the literature, and it seems impossible for us to list all the references. We refer the reader to \cite{asy1-3,15,1,8,13,12,19,2,14,4,7,wei-jmaa} and references therein.
It is well known that solutions of \eqref{p_p} are critical points of the functional $J_p: H_0^1(\om)\rt\R$ given by
\begin{equation}\label{11}
    J_p(u):=\frac{1}{2}\int_{\om}|\nabla u|^2dx-\frac{1}{p+1}\int_{\om}|u|^{p+1}dx.
\end{equation}
A nontrivial solution $u$ of \eqref{p_p} is called \emph{a least energy  solution} if the corresponding energy $J_p(u)$ is the smallest among all nontrivial solutions. The asymptotic behavior of least energy solutions of \eqref{p_p} as $p\to\infty$ was studied in \cite{asy1-3,7,wei-jmaa}, where among other things, it was proved that
least energy solutions satisfy 
\[\lim_{p\to+\infty}p\int_{\Omega}|u_p|^{p+1}dx=8\pi e.\]
Recently, for high-energy positive solutions of \eqref{p_p}, i.e. under the condition
\[\limsup_{p\to+\infty}p\int_{\Omega}|u_p|^{p+1}dx\leq \beta\]
with $\beta>8\pi e$,
the asymptotic behavior has been obtained in a series of papers \cite{15,8,13,4}.

On the other hand, in contrast with positive solutions, the asymptotic behavior of nodal solutions is more difficult to study and remains large open.
Here a solution $u$ of \eqref{p_p} is called a \emph{nodal} solution, if $u^{\pm}\neq0$, where $u^{+}=\max\{u,0\}$ and $u^{-}=\min\{u,0\}$.
A nodal solution $u$ is called \emph{a least energy nodal solution} if the corresponding energy $J_p(u)$ is the smallest among all nodal solutions.

Our motivation of this paper comes from the seminal work of Grossi, Grumiau and Pacella \cite{2}, where they studied the asymptotic behavior of the so-called {\it low energy nodal solutions}, i.e. nodal solutions satisfying
\begin{align}\label{c1.8}\tag{A}
    \lim_{p\to\infty}p\int_{\Omega}|u_p|^{p+1}dx = 16\pi e.
\end{align}
They proved that least energy nodal solutions always satisfy \eqref{c1.8}, i.e. are low energy nodal solutions.
For a nodal solution $u_p$ of \eqref{p_p}, we denote $\Vert u_p\Vert_{\infty}=\max_{x\in \Omega}|u_p(x)|$, and let $x_p^+$ (resp. $x_p^-$) be a maximum (resp. minimum) point of $u_p$ in $\om$, i.e.
\begin{align}\label{c1.12}
    u_p(x_p^+)=\max_{x\in \Omega}u_p(x)=\Vert u_p^+ \Vert_\infty,\quad
u_p(x_p^-)=\min_{x\in \Omega}u_p(x)=-\Vert u_p^- \Vert_\infty.
\end{align}
Without loss of generality, we may always assume
\begin{align}\label{c1.12^*}
    \Vert u_p\Vert_{\infty}=\Vert u_p^+\Vert_{\infty},\quad\text{i.e.}\quad \Vert u_p^+ \Vert_\infty\geq \Vert u_p^- \Vert_\infty.
\end{align}
To introduce the main result of \cite{2}, we recall the well-known Kirchoff-Routh type function  $\Psi:\mathcal{M}=\Omega^2\backslash\{(x,y)\in \Omega^2 : x= y\}\rightarrow \mathbb{R}$ defined by
\begin{align}\label{eq-1-13^2}
     \Psi(a_1,a_2):=\sum\limits_{i=1}^2\Psi_i(a_1,a_2)
     =2G(a_1,a_2)-R(a_1)-R(a_2),
\end{align}
with
\[\Psi_i(a_1,a_2):=G(a_1,a_2)-R(a_i).\]
Here $G(x,\cdot)$ denote the Green function of $-\Delta$ in $\Omega$ with the Dirichlet boundary condition, i.e.
\begin{align}\label{eq-1-5^*}
     \begin{cases}
    -\Delta G(x,\cdot) =\delta_x\quad&\text{in}\quad \Omega,\\
    G(x,\cdot) =0 \quad&\text{on}\quad\partial\Omega,
    \end{cases}
\end{align}
where $\delta_x$ is the Dirac function, $H(x,y)$ is the regular part of $G(x,y)$ given by
\begin{align}\label{eq-1-6^*}
    G(x,y)=-\frac{1}{2\pi}\log|x-y|+H(x,y),
\end{align}
and $R(x):=H(x,x)$.

Grossi, Grumiau and Pacella \cite{2} proved the following interesting result.

\begin{Theorem}\cite{2}\label{thm-LE}
Let $(u_p)_{p>0}$ be a family of low energy nodal solutions of \eqref{p_p}. 
Assume that
\begin{align}\label{eq-B-1}\tag{B}
    p\left(\Vert u_p^+\Vert_{\infty}-\Vert u_p^-\Vert_{\infty}\right)=O(1)\quad\text{ as }\quad p\rt\iy,
\end{align}
then $\|u_p^{\pm}\|_\infty\to \sqrt{e}$ and
the following statements hold up to a subsequence.
\begin{itemize}
\item[(1)] There are $x^+, x^-\in \Omega$ with $x^+\neq x^-$ such that
\[pu_p(x)\to 8\pi\sqrt{e}(G(x,x^+)-G(x,x^-))\quad\text{in }C^2_{loc}(\overline{\Omega}\setminus\{x^+,x^-\}).\]
\item[(2)] The pair $(x^+, x^-)$ is a critical point of the Kirchoff-Routh type function  $\Psi$ defined in \eqref{eq-1-13^2}, i.e. $\nabla \Psi(x^+,x^-)=0$.
\end{itemize}
\end{Theorem}

\begin{remark}\
\begin{itemize}
    \item [(1)] 
There are counterexamples in \cite{1,14} that the condition \eqref{eq-B-1} fails for some nodal solutions of \eqref{p_p}. For example, let $v_p$ be the radial nodal solution of \eqref{p_p} such that the energy $J_p(v_p)$ is the smallest among all radial nodal solutions of \eqref{p_p}, then it was proved in \cite{14} that (assume without loss of generality that $v_p(0)>0$)
\[\lim_{p\to\infty}\|v_p^+\|_\infty \approx 2.46>\sqrt{e},\quad \lim_{p\to\infty}\|v_p^-\|_\infty \approx 1.17<\sqrt{e}.\]
\item[(2)]
 Grossi, Grumiau and Pacella \cite{2,14} conjectured that the condition \eqref{eq-B-1} should hold automatically for the least energy nodal solutions of \eqref{p_p}. As far as we know, this question still remains open.
\end{itemize}
\end{remark}

In our previous paper \cite{CCZ}, we gave an alternative proof of Theorem A without assuming the condition \eqref{eq-B-1}, but we still can not prove the condition \eqref{eq-B-1} there.
Our first result of this paper confirms Grossi, Grumiau and Pacella's conjecture \cite{2,14}.

\begin{theorem}\label{cor-1.6}
Let $(u_p)_{p>1}$ be a family of low energy nodal solutions of \eqref{p_p}. Then $p(\Vert u_p^+\Vert_{\infty}-\Vert u_p^-\Vert_{\infty})=O(1)$ as $p\to\infty$, i.e. the condition \eqref{eq-B-1} holds automatically not only for least energy nodal solutions but also for all low energy nodal solutions.
\end{theorem}

Comparing to our previous paper \cite{CCZ}, the new stragety of proving
Theorem \ref{cor-1.6} is to prove the sharper estimates of $\|u_p^{\pm}\|_\infty$ than $\|u_p^{\pm}\|_\infty\to \sqrt{e}$, as stated as follows. 

\begin{theorem}\label{th2}
Let $(u_p)_{p>1}$ be a family of low energy nodal solutions of \eqref{p_p} and let $x_p^+, x_p^-$ be as in \eqref{c1.12}. Then for any small constant $\delta>0$,
up to a subsequence,
\begin{align}
     u_p(x_p^+)=&\sqrt{e}\left(1-\frac{\log{p}}{p}+\frac{1}{p}\left(4\pi\Psi_1(x^*_p)+3\log{2}+2\right)\right)+O\left(\frac{1}{p^{2-\delta}}\right),\label{eq-1-17}\\
      u_p(x_p^-)=&-\sqrt{e}\left(1-\frac{\log{p}}{p}+\frac{1}{p}\left(4\pi\Psi_2(x^*_p)+3\log{2}+2\right)\right)+O\left(\frac{1}{p^{2-\delta}}\right),\label{eq-1-18}
\end{align}
where $\Psi_1$, $\Psi_2$ are defined in $\eqref{eq-1-13^2}$ and  $x^*_p=(x^+_p,x^-_p)$.
\end{theorem}

\begin{remark}
In their interesting work \cite{4}, Grossi, Ianni, Luo and Yan obtained the sharp estimates of local maximums for the multi-spikes positive solutions  of \eqref{p_p}, and Theorem \ref{th2} generalizes their result to low energy nodal solutions. 

One key observation of proving Theroem \ref{th2} is to prove the existence of $\rho>0$ independent of $p$ such that
\begin{equation}\label{upun}
u_p(x)>0\quad\text{for $|x-x_p^+|\leq \rho$},\quad u_p(x)<0\quad\text{for $|x-x_p^-|\leq \rho$},
\end{equation}
see Lemma \ref{guaguawa}. This fact shows that near the maximum point or the minimun point, low energy nodal solutions do not change sign and hence   locally behave as positive solutions, from which we can follow the ideas from \cite{4} to prove the sharp estimates; see Section 3.

Again it follows easily from \cite{1,14} that \eqref{upun} can not hold for some high energy nodal solutions of \eqref{p_p}.
\end{remark}

We will see that Theorem \ref{cor-1.6} is an easy consequence of Theorem \ref{th2}. As another application of Theorem \ref{th2}, we can obtain the sharp estimate of the energy $J_p(u_p)$.

\begin{corollary}\label{prop sharp-1}
Let $(u_p)_{p>1}$ be a family of low energy nodal solutions of \eqref{p_p} with the concentrate points $\{x^+, x^-\}$ as stated in Theorem \ref{thm-LE}. Then 
\begin{align}\label{eq-4-80}
    J_p(u_p)=\frac{8\pi e}{p}+\frac{8\pi e}{p^2}\left(-2\log{p}+6\log{2}-3\right)+\frac{32\pi^2e}{p^2}\Psi(x^*)+o\left(\frac{1}{p^2}\right),
\end{align}
where $\Psi$ is defined in \eqref{eq-1-13^2} and $x^*=(x^+,x^-)$.
\end{corollary}

\begin{remark} In their interesting work \cite{19}, Esposito, Musso and Pistoia constructed the existence of some low energy nodal solutions satisfying the estimate \eqref{eq-4-80}, via the well-known method of finite-dimensional reductions.
Here Corollary \ref{prop sharp-1} shows that any low energy nodal solutions must satisfy the estimate \eqref{eq-4-80}.
\end{remark}

As a consequence of Corollary \ref{prop sharp-1}, we have

\begin{corollary}\label{th-7-2}
Let $(u_p)_{p>1}$ be a family of least energy nodal solutions to $\eqref{p_p}$ with the concentrate points $\{x^+, x^-\}$ as stated in Theorem \ref{thm-LE}. Then $(x^+, x^-)$ is a minimum point of $\Psi$, i.e.
\begin{align}
    \Psi (x^+,x^-)=\min\limits_{\mathcal{M}} \Psi(x,y).
\end{align}
\end{corollary}

\begin{remark}
Esposito, Musso and Pistoia \cite{19} conjectured that the converse assertment of Corollary \ref{th-7-2} should be true, namely if a family of low energy nodal solutions $v_p$ of \eqref{p_p} concentrate at minimum points of $\Psi$, then $v_p$ is a least energy nodal solution for $p$ large. This conjecture seems challenging and remains open.
\end{remark}

Our final result is to prove the non-degeneracy of low  energy nodal solutions.

\begin{theorem}\label{thm-1.6}
Let $(u_p)_{p>1}$ be a family of low energy nodal solutions of \eqref{p_p} with the concentrate points $\{x^+,x^-\}$ as stated as in Theorem \ref{thm-LE}.
Suppose that $(x^+,x^-)$ is a non-degenerate critical point of $\Psi$ defined in $\eqref{eq-1-13^2}$. Then there exists $p_0>1$ such that $u_p$ is non-degenerate for any $p>p_0$, i.e. the linearized equation of \eqref{p_p}
\begin{align}\label{lineaaried}
     \begin{cases}
    -\Delta \xi_p =|u_p|^{p-1}\xi_p\quad&\text{in}\quad \Omega,\\
    \xi_p=0 \quad&\text{on}\quad\partial\Omega,
    \end{cases}
\end{align}
has only trivial solutions $\xi_p\equiv 0$.
\end{theorem}
\begin{remark}\
\begin{itemize}
    \item [(1)] Bartsch, Micheletti and Pistoia \cite{23} proved that $\Psi$ is a Morse function for \emph{most domains} $\Omega$ of class $C^{m+2,\alpha}$, for any $m\geq0$ and $0<\alpha<1$, which means that all critical points of $\Psi$ are non-degenerate.
    
\item [(2)] If $(x^+,x^-)$ is a degenerate critical point of $\Psi$, then $u_p$ may be degenerate in general; see Section 4 for a concrete example. Therefore, the condition of non-dengeracy of $(x^+,x^-)$ is necessary for Theorem \ref{thm-1.6}.
\end{itemize}
\end{remark}

The rest of this paper is organized as follows. In Section 2, we recall some basic facts proved in \cite{CCZ}, which are needed in later sections. In Section 3, we establish the sharp estimates of $\|u_p^{\pm}\|_{\infty}$ and prove Theorems \ref{cor-1.6}-\ref{th2} and Corollaries \ref{prop sharp-1}-\ref{th-7-2}. Finally in Section 4, we prove Theorem \ref{thm-1.6}.

\textbf{Natations}.
Throughout the paper, $C$ always denotes constants that are independent of $p$ (possibly different in different places). Conventionally, we use $o(1)$ to denote quantities that converge to $0$ (resp. use $O(1)$ to denote quantities that remain bounded) as $p\to\infty$. Denote $B_r(x):=\{y\in\R^2 : |y-x|<r\}$. For any $s>1$, we denote $\|u\|_s:=(\int_{\Omega}|u|^sdx)^{1/s}$.

\section{Preliminaries}
\label{section-2}

In the sequel, we always let $(u_p)_{p>1}$ be a family of low energy nodal solutions of \eqref{p_p}, without assuming the condition \eqref{eq-B-1}.
In this section, we briefly review some basic facts proved in \cite{CCZ}. The following statements are often in the sense of up to a subsequence. 
\begin{itemize}
\item[(P1)] By H\"{o}lder inequality,
\be\label{est-v-1}
	p\int_\Omega |u_p|^p
	\le \left(p\int_\Omega |u_p|^{p+1}\right)^{\frac{p}{p+1}}(p|\Omega|)^{\frac{1}{p+1}}
	\le C.
\ee
Furthermore,
there are $x^+, x^-\in \Omega$ with $x^+\neq x^-$ such that
\begin{equation}\label{pup1}pu_p(x)\to 8\pi\sqrt{e}(G(x,x^+)-G(x,x^-))\quad\text{in}\quad C^2_{loc}(\overline{\Omega}\setminus\{x^+,x^-\}).\end{equation}
In particular, for any compact subset $K\subset \overline{\Omega}\setminus\{x^+,x^-\}$, there is $C>0$ independent of $p$ such that for $p$ large,
\begin{equation}\label{pup5}p|u_p(x)|+p|\nabla u_p(x)|\leq C,\quad \forall x\in K.\end{equation}
\item[(P2)] Recalling that $x_p^+$ (resp. $x_p^-$) is a maximum (resp. minimum) point of $u_p$ in $\om$, i.e.
\begin{align}\label{c1.12-2}
    u_p(x_p^+)=\max_{x\in \Omega}u_p(x)=\Vert u_p^+ \Vert_\infty,\quad
u_p(x_p^-)=\min_{x\in \Omega}u_p(x)=-\Vert u_p^- \Vert_\infty,
\end{align}
we have 
\be\label{pup8}\lim_{p\to\infty}x_p^{\pm}=x^{\pm},\quad \lim_{p\to\infty}u_p(x_p^{\pm})=\pm\sqrt{e}.\end{equation}
For convenience we also write 
\[x_{1,p}=x_p^+,\quad x_{2,p}=x_{p}^{-},\quad x_{1}=x^+,\quad x_{2}=x^{-}.\]

\item[(P3)] Define
\begin{align}\label{pup12}
\e_{i,p}>0,\quad   \e_{i,p}^{-2}:= p|u_p(x_{i,p})|^{p-1}\rightarrow +\infty,
\end{align}
\begin{align}\label{use 22}
    v_{i,p}(x):=\frac{p}{u_p(x_{i,p})}\left(u_p(x_{i,p}+\e_{i,p} x)-u_p(x_{i,p})\right),\quad\quad  i=1,2.
\end{align}
Then
\be\label{pup7}
\begin{cases}
  -\dl v_{i,p}(x)=\left | 1+\frac{v_{i,p}(x)}{p}\right | ^{p-1}\left(1+\frac{v_{i,p}(x)}{p}\right),\quad x\in \Omega_{i,p}:=\frac{\Omega-x_{i,p}}{\varepsilon_{i,p}},\\
  v_{i,p}(x)\leq v_{i,p}(0)=0,\\
  v_{i,p}(x)=-p,\quad x\in \partial \Omega_{i,p},
\end{cases}
\ee
and $v_{i,p}(x)\to U(x)$ 
in $C^2_{loc}(\mathbb{R}^2)$ as $p \rightarrow\infty$, where 
\begin{align}\label{define use}
    U(x)=\log{\frac{1}{(1+\frac{1}{8}|x|^2)^2}}
\end{align}
is the solution of $-\Delta U=e^U$ in $\mathbb{R}^2$, $U\leq 0$, $U(0)=0$ and $\int_{\mathbb{R}^2}e^U=8\pi$. Furthermore,
\begin{align}\label{pup16}
\int_{\Omega_{i,p}}\left|1+\frac{v_{i,p}(y)}{p}\right|^{p} dy
=\frac{1} {|u(x_{i,p})|}p \int_{\Omega} |u_p(x)|^p dx\leq C.
\end{align}

\item[(P4)] Take $r_0>0$ small such that $B_{2r_0}(x_1)\cap B_{2r_0}(x_2)=\emptyset$. Then for any $\varrho\in(0,r_0)$, we have 
\begin{align*}
     \lim\limits_{p\rightarrow \infty} p\int_{B_{\varrho}(x_1)}|u_p^{+}|^{p+1}= \lim\limits_{p\rightarrow \infty}p\int_{B_{\varrho}(x_2)}|u_p^{-}|^{p+1}=8\pi e,
\end{align*}
\begin{align}\label{pup4}
   \lim\limits_{p\rightarrow \infty} p\int_{B_{\varrho}(x_1)}|u_p^{-}|^{p+1}= \lim\limits_{p\rightarrow \infty}p\int_{B_{\varrho}(x_2)}|u_p^{+}|^{p+1}=0.
\end{align}


\end{itemize}

\begin{corollary}\label{Lemma-3-8} Consider the nodal line
\[NL_p:=\{x\in \Omega : u_p(x)=0\}.\]
Then $\overline{NL_p}\cap \partial\Omega\neq \emptyset$ for $p$ large.
\end{corollary}

\begin{proof}
Assume by contradiction that $\overline{NL_p}\cup \partial\Omega= \emptyset$, then the Hopf Lemma implies that $\frac{\partial u_p}{\partial \nu}$ does not change sign on $\partial\Omega$, where $\nu$ denotes the outer normal vector of $\partial\Omega$. Since $x_1, x_2\in\Omega$ and $x_1\neq x_2$, we have
\begin{align}
   \int_{\partial\Omega} \frac{\partial }{\partial \nu}\left(G(x,x_1)-G(x,x_2)\right)d\sigma= \int_{\Omega} \left(\Delta G(x,x_1)-\Delta G(x,x_2)\right)dx=0,
\end{align}
namely $\frac{\partial }{\partial \nu}\left(G(\cdot,x_1)-G(\cdot,x_2)\right)$ changes sign on $\partial\Omega$.
Since \eqref{pup1} implies
\begin{align}
    p\frac{\partial u_p}{\partial \nu}\rightarrow 8\pi \sqrt{e}\frac{\partial }{\partial \nu}\left(G(\cdot,x_1)-G(\cdot,x_2)\right),\quad\text{uniformly on }\partial\Omega \text{ as } p\rightarrow \infty,
\end{align}
so $\frac{\partial u_p}{\partial \nu}$ changes sign on $\partial\Omega$ for $p$ large enough, a contradiction.
\end{proof}

We recall the following results for later usage.

\begin{Lemma}\cite[Lemma 2.1]{7}\label{prop22}
Let $D\subset \mathbb{R}^2$ be a smooth bounded domain. Then for every $p>1$ there exists $S_p>0$ such that 
\begin{align}
    \left\Vert v \right\Vert_{L^{p+1}(D)}\leq S_p (p+1)^{\frac{1}{2}}\left\Vert \nabla v \right\Vert_{L^{2}(D)},\quad \forall v\in H_0^1 (D).
\end{align}
Moreover,
\begin{align}
    \lim\limits_{p\rightarrow \infty}S_p={(8\pi e)^{-\frac{1}{2}}}.
\end{align}
\end{Lemma}

\begin{Lemma}\cite[Lemma 2.3]{11}\label{Linear repe}
Recalling $U(z)$ defined in \eqref{define use},
let $u(z)$ be a solution of the linearized equation 
\begin{align}\label{linear U}
     \begin{cases}
    -\Delta u(z) =e^{U(z)} u(z)    \quad&\text{in}\quad \mathbb{R}^2,\\
    |u(z)|\leq C(1+|z|)^{\tau} \quad&\text{in}\quad \mathbb{R}^2,
    \end{cases}
\end{align}
for some $\tau\in[0,1)$. Then 
\begin{align}
    u(z)=\sum\limits_{i=1}^3 c_i \psi_i(z),
\end{align}
where $c_i\in\mathbb{R}$ are constants and $z=(z_1,z_2)$,
\begin{align}
    \psi_1(z)=\frac{z_1}{8+|z|^2}  ,\quad \psi_2(z)=\frac{z_2}{8+|z|^2} ,\quad\psi_3(z)=\frac{8-|z|^2}{8+|z|^2}.
\end{align}
\end{Lemma}

\section{Sharp estimates of the low energy  nodal solutions}
\label{section-4}

This section is devoted to the proof of Theorems \ref{cor-1.6}-\ref{th2} and Corollaries \ref{prop sharp-1}-\ref{th-7-2}. One key observation is as follows.

\begin{Lemma}\label{guaguawa}
  There exist $p_0>1$ and $\varrho_0 \in (0,r_0/2)$ such that for any $p>p_0$, up to a subsequence if necessary, $u_p(x)\geq 1/p$ in $B_{2\varrho_0}(x_1)$ and $u_p(x)\leq -1/p$ in $B_{2\varrho_0}(x_2)$.
\end{Lemma}

\begin{proof}
Clearly there exists $\varrho_0\in(0,r_0/2)$ such that
\begin{align}
    8\pi \sqrt{e}(G(x,x_1)-G(x,x_2))\geq 2,\quad \forall x\in \overline{B_{2\varrho_0}(x_1)}.
\end{align}
Then by \eqref{pup1}, there exists $p_1>1$ such that for any $p>p_1$,
\begin{align}\label{eq-3-33}
    p u_p(x)\geq 8\pi \sqrt{e}(G(x,x_1)-G(x,x_2))-1\geq 1,\quad\forall x\in \overline{B_{2\varrho_0}(x_1)}\backslash B_{\varrho_0}(x_1),
\end{align}
so
\begin{align}\label{eq-3-33-3}
u_p(x)>0,\quad\forall x\in \overline{B_{2\varrho_0}(x_1)}\backslash B_{\varrho_0}(x_1).
\end{align}
Then $u_p^-\chi_{B_{2\varrho_0}(x_1)}\in H_0^1(B_{\varrho_0}(x_1))\subset H_0^1(\Omega)$, where 
$\chi_{B_{2\varrho_0}(x_1)}$ denotes the characteristic function of $B_{2\varrho_0}(x_1)$, i.e.
\[\chi_{B_{2\varrho_0}(x_1)}(x)=\begin{cases}
1&\text{if }x\in B_{2\varrho_0}(x_1)\\
0&\text{if }x\not\in B_{2\varrho_0}(x_1).
\end{cases}\]
Consequently,
\[\int_{\Omega}\nabla u_p \nabla (u_p^-\chi_{B_{2\varrho_0}(x_1)}) =\int_{\Omega} |u_p|^{p-1}u_p u_p^-\chi_{B_{2\varrho_0}(x_1)},\]
i.e.
\begin{align}\label{xin1}
    \int_{B_{\varrho_0}(x_1)}|\nabla u_p^-|^2=\int_{B_{\varrho_0}(x_1)}|u_p^-|^{p+1}.
\end{align}
Assume by contradiction that up to a subsequence of $p>p_1$, 
\[ \int_{B_{\varrho_0}(x_1)}|u_p^-|^{p+1}>0.\]
Then Lemma \ref{prop22} and \eqref{xin1} imply
\begin{align*}
    0<\left(\int_{B_{\varrho_0}(x_1)}|u_p^-|^{p+1}\right)^{\frac{2}{p+1}}\leq& S_p^2(p+1)\int_{B_{\varrho_0}(x_1)}|\nabla u_p^-|^2\notag\\
    =&S_p^2(p+1)\int_{B_{\varrho_0}(x_1)}|u_p^-|^{p+1},
\end{align*}
and so
\begin{align*}
    p\int_{B_{\varrho_0}(x_1)}|u_p^-|^{p+1}\geq \frac{p}{(S_p^2(p+1))^{\frac{p+1}{p-1}}}\geq 8\pi e+o(1),\quad\text{as }p\to\infty,
\end{align*}
which contradicts with \eqref{pup4}. 

Therefore, there is $p_0>p_1$ such that for any $p>p_0$, we have $\int_{B_{\varrho_0}(x_1)}|u_p^-|^{p+1}=0$, i.e. $u_p(x)\geq 0$ in $ B_{\varrho_0}(x_1)$. From here and \eqref{eq-3-33-3}, we see from the strong maximum principle that $u_p(x)> 0$ in $ B_{2\varrho_0}(x_1)$. In fact, \eqref{eq-3-33} and the strong maximum principle imply that $pu_p(x)\geq 1$ in $ B_{2\varrho_0}(x_1)$.

Similarly we can prove $pu_p(x)\leq -1$ in $B_{2\varrho_0}(x_2)$ for $p>p_0$ by taking $\varrho_0$ smaller and $p_0$ larger if necessary.
\end{proof}

Recalling $v_{i,p}$ and $\varepsilon_{i,p}$ defined in \eqref{use 22}, it follows from Lemma \ref{guaguawa} and \eqref{c1.12-2}-\eqref{use 22} that for $p$ large,
\begin{align}\label{hao chu}
    0< 1+\frac{v_{i,p}(z)}{p}=\frac{u_p(\varepsilon_{i,p}z+x_{i,p})}{u_p(x_{i,p})}\leq 1 \quad \text{ in } \,\,B_{\frac{2\varrho_0}{\varepsilon_{i,p}}}(0),\quad i=1,2.
\end{align}

\begin{Lemma}\label{prop44}
For any $\delta\in(0,1)$, there exist $p_{\delta}>p_0$, $R_{\delta}>1$ and $C_\delta>0$ such that 
\begin{align}\label{pup6}
   v_{i,p}(z)\leq \left(4 -\delta \right)\log{\frac{1}{|z|}}+ C_{\delta}, \quad\quad \text{ for } i=1,2,
\end{align}
provided $R_{\delta}\leq |z| \leq \frac{\varrho_0}{\varepsilon_{i,p}}$ and $p>p_{\delta}$. 
\end{Lemma}

\begin{proof} Let $ |z| \leq \frac{\varrho_0}{\varepsilon_{i,p}}$. 
Inserting \eqref{hao chu} into the Green representation formula 
\[u_p(\varepsilon_{i,p}z+x_{i,p})=\int_{\Omega}G(\varepsilon_{i,p}z+x_{i,p},y)|u_p(y)|^{p-1}u_p(y)dy\]
and recalling \eqref{pup7} that $\Omega_{i,p}=\frac{\Omega-x_{i,p}}{\varepsilon_{i,p}}$, we easily obtain
{\allowdisplaybreaks
\begin{align}\label{ct1}
    v_{i,p}(z)=&v_{i,p}(z)-v_{i,p}(0)\notag\\
    =&\int_{\Omega_{i,p}}\left[G(x_{i,p}+\varepsilon_{i,p} z, x_{i,p}+\varepsilon_{i,p} y)-G(x_{i,p}, x_{i,p}+\varepsilon_{i,p} y)\right]\notag\\
    &\left|1+\frac{v_{i,p}(y)}{p}\right|^{p-1}\left(1+\frac{v_{i,p}(y)}{p}\right) dy\notag\\
    =&\frac{1}{2\pi}\int_{B_{\frac{2\varrho_0}{\varepsilon_{i,p}}}(0)} \log{\frac{|y|}{|z-y|}}\left|1+\frac{v_{i,p}(y)}{p}\right|^{p-1}\left(1+\frac{v_{i,p}(y)}{p}\right)  dy\\
    &+\frac{1}{2\pi}\int_{\Omega_{i,p}\setminus B_{\frac{2\varrho_0}{\varepsilon_{i,p}}}(0)} \log{\frac{|y|}{|z-y|}}\left|1+\frac{v_{i,p}(y)}{p}\right|^{p-1}\left(1+\frac{v_{i,p}(y)}{p}\right)  dy\notag\\
&+\int_{\Omega_{i,p}}\left[H(x_{i,p}+\varepsilon_{i,p} z, x_{i,p}+\varepsilon_{i,p} y)-\right.\notag\\
    &\left. H(x_{i,p}, x_{i,p}+\varepsilon_{i,p} y)\right]
   \left|1+\frac{v_{i,p}(y)}{p}\right|^{p-1}\left(1+\frac{v_{i,p}(y)}{p}\right) dy\nonumber\\
   =&I_1(z)+I_2(z)+I_3(z).\nonumber
\end{align}
}%
Since $H(\cdot, \cdot)$ is bounded in $B_{2\varrho_0}(x_i)\times \Omega$, it follows from \eqref{pup16} that
\begin{align}\label{ctttt1}
|I_3(z)|\leq C \int_{\Omega_{i,p}}\left|1+\frac{v_{i,p}(y)}{p}\right|^{p} dy\leq C.
\end{align}
Since $|z| \leq \frac{\varrho_0}{\varepsilon_{i,p}}$, we have
$\frac{2}{3}\leq \frac{|y|}{|z-y|}\leq 2$ for $|y|\geq \frac{2\varrho_0}{\varepsilon_{i,p}}$,
so
\begin{align}\label{eq-4-7^*}
    |I_2(z)|\leq C \int_{\Omega_{i,p}\setminus B_{\frac{2\varrho_0}{\varepsilon_{i,p}}}(0)}\left|1+\frac{v_{i,p}(y)}{p}\right|^{p} dy
\leq C.
\end{align}
On the other hand, by (P2) in Section 2, the same proof as \cite[Lemma 4.4]{8} (where postive solutions of \eqref{p_p} was studied) implies that for any $\delta\in (0,1)$, there exsits $R_\delta>0$, $p_\delta>p_0$ and $\widetilde{C}_{\delta}>0$ such that for any $R_{\delta}\leq |z|\leq \frac{\varrho_0}{\varepsilon_{i,p}}$ and $p>p_\delta$, there holds
\[\frac{1}{2\pi}\int_{B_{\frac{2\varrho_0}{\varepsilon_{i,p}}}(0)} \log{\frac{|y|}{|z-y|}}\left|1+\frac{v_{i,p}(y)}{p}\right|^{p}dy\leq \left(4 -\delta \right)\log{\frac{1}{|z|}}+ \widetilde{C}_{\delta}.\]
Now thanks to \eqref{hao chu}, we have
\begin{equation}\label{pup9}I_1(z)=\frac{1}{2\pi}\int_{B_{\frac{2\varrho_0}{\varepsilon_{i,p}}}(0)} \log{\frac{|y|}{|z-y|}}\left|1+\frac{v_{i,p}(y)}{p}\right|^{p}dy\leq \left(4 -\delta \right)\log{\frac{1}{|z|}}+ \widetilde{C}_{\delta}.\end{equation}
Inserting \eqref{ctttt1}, \eqref{eq-4-7^*} and \eqref{pup9} into \eqref{ct1}, we obtain the desired estimate \eqref{pup6}. \end{proof}

\begin{remark}
Remark that in the above proof, without \eqref{hao chu} we can not say
\[I_1(z)\leq\frac{1}{2\pi}\int_{B_{\frac{2\varrho_0}{\varepsilon_{i,p}}}(0)} \log{\frac{|y|}{|z-y|}}\left|1+\frac{v_{i,p}(y)}{p}\right|^{p}dy\] and so can not obtain the desired upper bound for $I_1(z)$. That is, Lemma \ref{guaguawa}   plays a crucial role in Lemma \ref{prop44}. As mentioned before, we can not expect the validity of Lemma \ref{guaguawa}  for general nodal solutions.
\end{remark}

\begin{corollary}\label{re52}
For any $\delta\in(0,1)$, let $p_{\delta}$, $R_{\delta}$ be as in Lemma \ref{prop44}. Then exists $C_\delta>0$ such that 
\begin{align}\label{pup10}
    \left|1+\frac{v_{i,p}(z)}{p}\right|^{p}\leq  \frac{C_\delta}{1+|z|^{4-\delta}}, \quad i=1,2,
\end{align}
provided $|z| \leq \frac{\varrho_0}{\varepsilon_{i,p}}$ and $p>p_{\delta}$. 
\end{corollary}

\begin{proof}
By \eqref{hao chu} and Lemma \ref{prop44}, we have that for $R_{\delta}\leq |z|\leq \frac{\varrho_0}{\varepsilon_{i,p}}$,
\begin{align}
    \left|1+\frac{v_{i,p}(z)}{p}\right|^{p}=e^{p\log\left(1+\frac{v_{i,p}(z)}{p}\right)}\leq e^{v_{i,p}(z)}\leq  \frac{C}{|z|^{4-\delta}},
\end{align}
for $p>p_\delta$. Since \eqref{hao chu} also implies $|1+\frac{v_{i,p}(z)}{p}|^{p}\leq 1$ for $|z|\leq R_{\delta}$, we obtain \eqref{pup10}.
\end{proof}

Thanks to Corollary \ref{re52}, the dominated convergent theorem can be applied to obtain refined estimates; see below.

\begin{Lemma}\label{B.l1} For any fixed $\varrho\in (0,\frac{\varrho_0}{2}]$, we 
define
\begin{align}\label{pup43}
    C_{i,p}:=\int_{B_{\varrho}(x_{i,p})}|u_p(x)|^p dx=\frac{|u_p(x_{i,p})|}{p}\int_{B_{\frac{\varrho}{\varepsilon_{i,p}}}(0)}\left|1+\frac{v_{i,p}(z)}{p}\right|^p dz=O\left(\frac{1}{p}\right).
\end{align}
Then
\begin{align}\label{eqA-1-1}
    u_p(x)=C_{1,p}G(x_{1,p},x)-C_{2,p}G(x_{2,p},x)+o\left(\frac{\varepsilon_{p}}{p}\right), 
\end{align}
in $C^1(\overline{\Omega}\backslash\cup_{i=1}^2 B_{2\varrho}(x_{i,p}))$, where $\varepsilon_p=\max\{\varepsilon_{1,p}, \varepsilon_{2,p}\}$.
\end{Lemma}
\begin{proof}
Thanks to Corollary \ref{re52}, the proof is similar to that of \cite[Lemma 3.3]{4}. 
For $x\in \Omega\backslash\cup_{i=1}^2 B_{2\varrho}(x_{i,p})$, we have
\begin{align}\label{eqA-12}
    u_p(x)=&\int_{\Omega}G(y,x)|u_p(y)|^{p-1}u_p(y)dy\notag\\
    =& \int_{B_{\varrho}(x_{1,p})}G(y,x)|u_p(y)|^{p}dy-\int_{B_{\varrho}(x_{2,p})}G(y,x)|u_p(y)|^{p}dy\\
   &+\int_{\Omega\backslash(B_{\varrho}(x_{1,p})\cup B_{\varrho}(x_{2,p}))}G(y,x)|u_p(y)|^{p-1}u_p(y)dy.\notag
\end{align}
By the Taylor expansion and Corollary \ref{re52}, we have
{\allowdisplaybreaks
\begin{align}\label{eqA-14}
    &\int_{B_{\varrho}(x_{i,p})}G(y,x)|u_p(y)|^{p}dy-C_{i,p}G(x_{i,p},x)\notag\\
    =&\int_{B_{\varrho}(x_{i,p})}(G(y,x)-G(x_{i,p},x))|u_p(y)|^{p}dy\notag\\
    =&\frac{|u_p(x_{i,p})|}{p}\int_{B_{\frac{\varrho}{\varepsilon_{i,p}}}(0)}(G(\varepsilon_{i,p}z+x_{i,p},x)-G(x_{i,p},x))\left|1+\frac{v_{i,p}(z)}{p}\right|^p dz\notag\\
    =&\frac{\varepsilon_{i,p}|u_p(x_{i,p})|}{p}\sum_{j=1}^2\frac{\partial}{\partial y_j}G(x_{i,p},x)\int_{B_{\frac{\varrho}{\varepsilon_{i,p}}}(0)}z_j\left|1+\frac{v_{i,p}(z)}{p}\right|^p dz\notag\\
    &+\frac{\varepsilon_{i,p}^2|u_p(x_{i,p})|}{p}O\left(\int_{B_{\frac{\varrho}{\varepsilon_{i,p}}}(0)}|z|^2\left|1+\frac{v_{i,p}(z)}{p}\right|^p dz\right)\notag\\
    =&o\left(\frac{\varepsilon_{i,p}}{p}\right),
\end{align}
}%
where in the last equality we used 
\[\lim_{p\to\infty}\int_{B_{\frac{\varrho}{\varepsilon_{i,p}}}(0)}z_j\left|1+\frac{v_{i,p}(z)}{p}\right|^p dz=\int_{\R^2}z_j e^{U(z)}dz=0\]
by the dominated convergent theorem and \eqref{pup7}-\eqref{define use}.

Since \eqref{pup5} and \eqref{pup12} also imply
\begin{align}\label{eqA-15}
 \int_{\Omega\backslash(B_{\varrho}(x_{1,p})\cup B_{\varrho}(x_{2,p}))}G(y,x)|u_p(y)|^{p-1}u_p(y)dy=O\left(\frac{C^p}{p^p}\right)=o\left(\frac{\varepsilon_{p}}{p}\right),
\end{align}
we obtain \eqref{eqA-1-1} in $C(\overline{\Omega}\backslash\cup_{i=1}^2 B_{2\varrho}(x_{i,p}))$. A similar argument shows that \eqref{eqA-1-1} holds in $C^1(\overline{\Omega}\backslash\cup_{i=1}^2 B_{2\varrho}(x_{i,p}))$.
\end{proof}

\begin{Lemma}\label{lemma asy}
  There exists $C>0$ such that for $p$ large enough, 
  \begin{align}
      |v_{i,p}(z)| \leq C\log(2+|z|),\quad \text{for }|z|\leq \frac{\varrho_0}{\varepsilon_{i,p}}.
  \end{align}
\end{Lemma}

\begin{proof} 
Since $v_{i,p}\rightarrow U$ in $C^2_{loc}(\mathbb{R}^2)$, there exists $C>0$ such that $|v_{i,p}(z)|\leq C$ for $|z|\leq 2$ for $p$ large enough.  
Now consider $2\leq|z|\leq \varrho_0/\varepsilon_{i,p}$. 
Then by \eqref{ct1}-\eqref{eq-4-7^*},
\begin{align}\label{l53-1}
    v_{i,p}(z)=\frac{1}{2\pi}\int_{|y|\leq \frac{2\varrho_0}{\varepsilon_{i,p}}} \log{\frac{|y|}{|z-y|}}\left|1+\frac{v_{i,p}(y)}{p}\right|^{p} dy+O(1).
\end{align}
By \eqref{hao chu} and \eqref{pup16}, we have
\begin{align}\label{pup17}
&\left|\int_{|y|\leq \frac{2\varrho_0}{\varepsilon_{i,p}}} \log\frac{|y|}{|y|+1}\left|1+\frac{v_{i,p}(y)}{p}\right|^{p}dy\right|\nonumber\\
\leq &C\int_{|y|\leq 2} \left|\log\frac{|y|}{|y|+1}\right|dy+
C\int_{|y|\leq \frac{2\varrho_0}{\varepsilon_{i,p}}, |y|> 2} \left|1+\frac{v_{i,p}(y)}{p}\right|^{p}dy\leq C,
\end{align}
\begin{align}\label{l53-2}
   &\left| \int_{|y|\leq \frac{{2\varrho_0}}{{\varepsilon_{i,p}}}, |z-y|\leq 1} \log{\frac{|y|+1}{|z-y|}}\left|1+\frac{v_{i,p}(y)}{p}\right|^{p} dy\right|\notag\\
   \leq & C\int_{|z-y|\leq 1} \log\frac{1}{|z-y|}dy+C\log(|z|+2) \int_{|y|\leq \frac{{2\varrho_0}}{{\varepsilon_{i,p}}}, |z-y|\leq 1} \left|1+\frac{v_{i,p}(y)}{p}\right|^{p} dy\notag\\
    \leq & C\log (2+|z|).
\end{align}
Note that for $|z-y|\geq 1$ and $|z|\geq 2$, we have (see e.g. \cite[Lemma 8.2.3]{10})
\[
    \left|\frac{\log|z-y|-\log|z|-\log(|y|+1)}{\log|z|}\right|\leq C,
\]
i.e. 
$|\log{\frac{|y|+1}{|z-y|}}|\leq C \log|z|
$, 
so
\begin{align}\label{l53-3}
    &\left| \int_{|y|\leq \frac{{2\varrho_0}}{{\varepsilon_{i,p}}}, |z-y|\geq 1} \log{\frac{|y|+1}{|z-y|}}\left|1+\frac{v_{i,p}(y)}{p}\right|^{p} dy\right|\notag\\
    \leq &C\log|z|\int_{|y|\leq \frac{{2\varrho_0}}{{\varepsilon_{i,p}}}, |z-y|\geq 1} \left|1+\frac{v_{i,p}(y)}{p}\right|^{p} dy
    \leq  C\log |z|.
\end{align}
The proof is complete by inserting \eqref{pup17}-\eqref{l53-3} into \eqref{l53-1}.
\end{proof}

\begin{Proposition}\label{prop54}
Define \[w_{i,p}:=p(v_{i,p}-U).\] Then for any fixed $\tau\in(0,1)$, there exists $C_{\tau}>0$ such that
\begin{align}\label{eq-4-37}
    |w_{i,p}(z)|\leq C_{\tau} \left(1+|z|\right)^{\tau} \quad\quad\text{in}\quad B_{\frac{\varrho_0}{\varepsilon_{i,p}}}(0),\quad i=1,2.
\end{align}
Consequently, $w_{i,p}\rightarrow w_0$ in $C^{2}_{loc}(\mathbb{R}^2)$ as $p\rightarrow\infty$, where $w_0$ solves the non-homogeneous linear equation
\begin{align}\label{re-eq1}
    -\Delta w_0-e^U w_0=-\frac{U^2}{2}e^U\quad\quad\text{in}\quad\mathbb{R}^2,
\end{align}
and for any $\tau\in(0,1)$,
\begin{align}\label{re-eq2}
    |w_0(z)|\leq C_{\tau} \left(1+|z|\right)^{\tau}, \quad z\in\mathbb{R}^2.
\end{align}
Moreover,
\begin{align}
    \int_{\mathbb{R}^2}\Delta w_0 dz=24\pi.
\end{align}

\end{Proposition}

\begin{proof}
Thanks to \eqref{hao chu} and the previous estimates as stated in Lemmas \ref{prop44}-\ref{lemma asy}, 
the proof is the same as \cite[Proposition 3.5]{4} where the same statements were proved for positive solutions of \eqref{p_p}, so we omit the details here.
\end{proof}

\begin{Proposition}\label{prop56}
Define \[k_{i,p}:=p(w_{i,p}-w_0),\] where $w_{i,p}$ and $w_0$ are as in Proposition \ref{prop54}. Then for any fixed $\tau\in (0,1)$, there exists $C_{\tau}>0$ such that
\begin{align}\label{eq-4-10-1}
    |k_{i,p}(z)|\leq C \left(1+|z|\right)^{\tau} \quad\quad\text{in}\quad B_{\frac{\varrho_0}{\varepsilon_{i,p}}}(0),\quad i=1,2.
\end{align}
Consequently,
\begin{align}
    v_{i,p}=U+\frac{w_0}{p}+\frac{k_{i,p}}{p^2}=U+\frac{w_0}{p}+\frac{O(1)}{p^2}\quad\quad\text{in}\quad C^2_{loc}(\mathbb{R}^2).
\end{align}
\end{Proposition}

\begin{proof}
Again the proof is same as \cite[Proposition 3.9]{4}.
\end{proof}

\begin{Proposition}\label{prop57} Fix any $\varrho\in (0,\frac{\varrho_0}{2}]$. Then
for any given small $\delta>0$ and $i=1,2$,
\begin{align}
    \int_{B_{\frac{\varrho}{\varepsilon_{i,p}}}(0)}
        \left|1
        +\frac{v_{i,p}(z)}{p}\right|^p dz=8\pi-\frac{24\pi}{p}+O\left(\frac{1}{p^{2-\delta}}\right),\label{sharp533}
\end{align}
        \begin{align}
    \int_{B_{\frac{\varrho}{\varepsilon_{i,p}}}(0)}
        \log{|z|}\left|1
        +\frac{v_{i,p}(z)}{p}\right|^p dz=12\pi\log{2}+O\left(\frac{1}{p}\right),\label{sharp534}
        \end{align}
        \begin{align}
    \int_{B_{\frac{\varrho}{\varepsilon_{i,p}}}(0)}
        \left|1
        +\frac{v_{i,p}(z)}{p}\right|^{p+1} dz=8\pi-\frac{40\pi}{p}+O\left(\frac{1}{p^{2-\delta}}\right).\label{sharp535^*}
        \end{align}

\end{Proposition}

\begin{proof} The proof is similar to that of \cite[Proposition 3.10]{4} where \eqref{sharp533}-\eqref{sharp534} were proved for positive solutions. 
Note from the definition \eqref{pup12} of $\varepsilon_{i,p}$ that $p\ll \varrho/\varepsilon_{i,p}$ for $p$ large enough. Denote
\[R(p,\frac{\varrho}{\varepsilon_{i,p}}):=\{y \in \mathbb{R}^2 :p \leq|y|\leq \frac{\varrho}{\varepsilon_{i,p}}\}.\] 
By \eqref{hao chu} and Corollary \ref{re52}, we have 
\begin{align}\label{sharp535}
    \int_{R(p,\frac{\varrho}{\varepsilon_{i,p}})}\left|1+\frac{v_{i,p}}{p}\right|^{p+1} dz&\leq\int_{R(p,\frac{\varrho}{\varepsilon_{i,p}})}\left|1+\frac{v_{i,p}}{p}\right|^p dz\\
        &=O\left(\int_{R(p,\frac{\varrho}{\varepsilon_{i,p}})}\frac{dz}{1+|z|^{4-\delta}}\right)=O\left(\frac{1}{p^{2-\delta}}\right),\nonumber
 \end{align}
 \begin{align}\label{sharp535^*1}
    \int_{R(p,\frac{\varrho}{\varepsilon_{i,p}})}\log{|z|}\left|1+\frac{v_{i,p}}{p}\right|^p dz
        =O\left(\int_{R(p,\frac{\varrho}{\varepsilon_{i,p}})}\frac{\log{|z|}dz}{1+|z|^{4-\frac{\delta}{2}}}\right)=O\left(\frac{1}{p^{2-\delta}}\right).
\end{align}
Recall that $v_{i,p}=U+w_{i,p}/{p}=U+w_0/p+k_{i,p}/p^2$.
By $\eqref{eq-4-37}$ we have
\begin{align}\label{eq-4-61}
    |v_{i,p}|=\left|U+\frac{w_{i,p}}{p}\right|=O\left(\log{p}+p^{\tau-1}\right)=O\left(\log{p}\right),\quad\text{for } |z|\leq p,
\end{align}
and \eqref{re-eq2}-\eqref{eq-4-10-1} imply (take $\tau=\frac{\delta}{2}$)
\begin{align}
    |w_0|=O\left(p^{\frac{\delta}{2}}\right),\quad |k_{i,p}|=O\left(p^{\frac{\delta}{2}}\right),\quad\text{for }  |z|\leq p.
\end{align}
These, together with \eqref{hao chu} and the Taylor expansion, yield that 
\begin{align}\label{sharp537}
    \left|1+\frac{v_{i,p}}{p}\right|^p=&\exp\left(p\log{\left(1+\frac{v_{i,p}}{p}\right)}\right)=\exp\left(v_{i,p}-\frac{v_{i,p}^2}{2p}+O\left(\frac{v_{i,p}^3}{p^2}\right)\right)\notag\\
    =&\exp\left(\left(U+\frac{w_0}{p}+\frac{k_{i,p}}{p^2}\right)-\frac{1}{2p}\left(U+\frac{w_0}{p}+\frac{k_{i,p}}{p^2}\right)^2+O\left(\frac{1}{p^{2-\delta}}\right)\right)\notag\\
    =&e^U\left(1+\frac{1}{p}\left(w_0-\frac{U^2}{2}\right)+O\left(\frac{1}{p^{2-\frac{\delta}{2}}}\right)\right),\quad\text{for } |z|\leq p.
\end{align}
Note from \eqref{define use} and \eqref{re-eq2} that
\begin{align}\label{sharp538}
    \int_{\mathbb{R}^2\backslash B_p(0)}e^U\left(1+\frac{1}{p}\left(w_0-\frac{U^2}{2}\right)\right)
    =&O\left(\int_{\mathbb{R}^2\backslash B_p(0)}\frac{dz}{|z|^{4-\delta}}\right)=O\left(\frac{1}{p^{2-\delta}}\right).
\end{align}
It follows from \eqref{sharp537}-\eqref{sharp538} and Proposition \ref{prop54} that
\begin{align}\label{sharp539}
    \int_{B_p(0)} \left|1+\frac{v_{i,p}}{p}\right|^p dz
    =&\int_{\mathbb{R}^2}e^{U}\left(1+\frac{1}{p}\left(w_0-\frac{U^2}{2}\right)\right) dz+O\left(\frac{1}{p^{2-\delta}}\right)\notag\\
    =&\int_{\mathbb{R}^2}e^{U}dz-\frac{1}{p}\int_{\mathbb{R}^2}\Delta w_0dz+O\left(\frac{1}{p^{2-\delta}}\right)\notag\\
    =&8\pi-\frac{24\pi}{p}+O\left(\frac{1}{p^{2-\delta}}\right).
\end{align}
Thus, \eqref{sharp533} follows from \eqref{sharp535} and \eqref{sharp539}.

By a similar argument as $\eqref{sharp539}$, we have 
\begin{align}
    &\int_{B_p(0)}\log{|z|}\left|1+\frac{v_{i,p}}{p}\right|^p dz\notag\\
    =&\int_{B_p(0)}\log{|z|}e^{U}\left(1+
    \frac{1}{p}\left(w_0-\frac{U^2}{2}\right)\right) dz+O\left(\frac{1}{p^{2-\delta}}\right)\notag\\
    =&\int_{\mathbb{R}^2}\log{|z|}e^{U}dz+
    \frac{1}{p}\int_{\mathbb{R}^2}\log{|z|}e^{U}\left(w_0-\frac{U^2}{2}\right) dz+O\left(\frac{1}{p^{2-\delta}}\right)\notag\\
    =&2\pi\int_0^{\infty}\frac{s\log{s}}{(1+\frac{1}{8}s^2)^2}ds+O\left(\frac{1}{p}\right)=12\pi\log{2}+O\left(\frac{1}{p}\right).
\end{align}
This together with $\eqref{sharp535^*1}$ implies $\eqref{sharp534}$. 

Using a similar argument as \eqref{sharp537}-\eqref{sharp538}, we have
\begin{align}
    \left|1+\frac{v_{i,p}}{p}\right|^{p+1}=e^U\left(1+\frac{1}{p}\left(U+w_0-\frac{U^2}{2}\right)+O\left(\frac{1}{p^{2-\frac{\delta}{2}}}\right)\right),\quad |z|\leq p,
\end{align}
\begin{align}
    \int_{\mathbb{R}^2\backslash B_p(0)}e^U\left(1+\frac{1}{p}\left(U+w_0-\frac{U^2}{2}\right)\right)dz
    =O\left(\frac{1}{p^{2-\delta}}\right),
\end{align}
so
\begin{align}
    \int_{B_p(0)}\left|1+\frac{v_{i,p}}{p}\right|^{p+1} dz
     =&\int_{\mathbb{R}^2}e^U\left(1+\frac{1}{p}\left(U+w_0-\frac{U^2}{2}\right)\right) dz+O\left(\frac{1}{p^{2-\delta}}\right)\notag\\
     =&\int_{\mathbb{R}^2}e^{U}dz+\frac{1}{p}\int_{\mathbb{R}^2} (U e^{U}-\Delta w_0) dz+O\left(\frac{1}{p^{2-\delta}}\right)\notag\\
     =&8\pi -\frac{40\pi}{p}+O\left(\frac{1}{p^{2-\delta}}\right),
\end{align}
where we used $\int_{\mathbb{R}^2}Ue^Udz=-16\pi$ to obtain the last equality.
This combining with \eqref{sharp535} implies \eqref{sharp535^*}. 
\end{proof}

Now we are in the position to prove Theorems \ref{cor-1.6}-\ref{th2} and Corollaries \ref{prop sharp-1}-\ref{th-7-2}.

\begin{proof}[Proof of Theorem \ref{th2}] Fix any $\varrho\in (0,\frac{\varrho_0}{2}]$. By \eqref{pup5} and Lemma \ref{guaguawa}, 
\begin{align}\label{sharp1}
    u_p(x_{1,p})=&\int_{\Omega}G(y,x_{1,p})|u_p(y)|^{p-1}u_p(y)dy\notag\\
    =& \int_{B_{\varrho}(x_{1,p})}G(y,x_{1,p})|u_p(y)|^{p}dy-\int_{B_{\varrho}(x_{2,p})}G(y,x_{1,p})|u_p(y)|^{p}dy\notag\\
   &+O\left(\frac{C^p}{p^p}\right).
\end{align}
Note that
\begin{align}\label{sharp542}
    &\int_{B_{{\varrho}}(x_{1,p})}G(y,x_{1,p})|u_p(y)|^{p}dy\notag\\
    =&\frac{u_p(x_{1,p})}{p} \int_{B_{\frac{\varrho}{\varepsilon_{1,p}}}(0)}G(x_{1,p}+\varepsilon_{1,p}z,x_{1,p})\left|1+\frac{v_{1,p}(z)}{p}\right|^p dz\notag\\
        =&\frac{u_p(x_{1,p})}{p} \int_{B_{\frac{\varrho}{\varepsilon_{1,p}}}(0)}
        H(x_{1,p}+\varepsilon_{1,p}z,x_{1,p})\left|1+\frac{v_{1,p}(z)}{p}\right|^p dz\notag\\
        &-\frac{u_p(x_{1,p})}{2\pi p} \int_{B_{\frac{\varrho}{\varepsilon_{1,p}}}(0)}
        \log{|z|}\left|1
        +\frac{v_{1,p}(z)}{p}\right|^p dz\notag\\
         &-\frac{u_p(x_{1,p})\log{\varepsilon_{1,p}}}{2\pi p} \int_{B_{\frac{\varrho}{\varepsilon_{1,p}}}(0)}
        \left|1
        +\frac{v_{1,p}(z)}{p}\right|^p dz,
\end{align}
and by Corollary \ref{re52}, 
\begin{align}\label{sharp543}
    &\int_{B_{\frac{\varrho}{\varepsilon_{1,p}}}(0)}
        H(x_{1,p}+\varepsilon_{1,p}z,x_{1,p})\left|1+\frac{v_{1,p}(z)}{p}\right|^p dz\notag\\
       =&H(x_{1,p},x_{1,p})\int_{B_{\frac{\varrho}{\varepsilon_{1,p}}}(0)}\left|1+\frac{v_{1,p}(z)}{p}\right|^p dz\notag\\
       &\quad\quad\quad\quad\quad\quad\quad\quad\quad+\varepsilon_{1,p}O\left(\int_{B_{\frac{\varrho}{\varepsilon_{1,p}}}(0)}
       |z|\left|1+\frac{v_{1,p}(z)}{p}\right|^p dz\right)\notag\\
       =&H(x_{1,p},x_{1,p})\int_{B_{\frac{\varrho}{\varepsilon_{1,p}}}(0)}
       \left|1+\frac{v_{1,p}(z)}{p}\right|^p dz+O(\varepsilon_{1,p}).
\end{align}
It follows from \eqref{sharp542}-\eqref{sharp543}, $O(\varepsilon_{1,p})=o(\frac1p)$ and Proposition \ref{prop57} that
\begin{align}\label{c4.66}
    \int_{B_{\varrho}(x_{1,p})}G(y,x_{1,p})|u_p(y)|^{p}dy
    =&\frac{u_p(x_{1,p})}{ p}\left(8\pi H(x_{1,p},x_{1,p})-6\log{2}+O\left(\frac{1}{p}\right)\right)\notag\\&-\frac{4u_p(x_{1,p})\log{\varepsilon_{1,p}}}{ p}\left(1-\frac{3}{p}+O\left(\frac{1}{p^{2-\delta}}\right)\right).
\end{align}
Similarly, we have (note $u_p(x_{2,p})<0$ and $G(x_{2,p},x_{1,p})=G(x_{1,p},x_{2,p})$)
\begin{align}\label{sharp545}
    &\int_{B_{\varrho}(x_{2,p})}G(y,x_{1,p})|u_p(y)|^{p}dy\notag\\
    =&-\frac{u_p(x_{2,p})}{p} \int_{B_{\frac{\varrho}{\varepsilon_{2,p}}}(0)}G(x_{2,p}+\varepsilon_{1,p}z,x_{1,p})\left|1+\frac{v_{2,p}(z)}{p}\right|^p dz\notag\\
    =&-\frac{u_p(x_{2,p})}{ p}\left(8\pi G(x_{1,p},x_{2,p})+O\left(\frac{1}{p}\right)\right).
\end{align}
Inserting \eqref{c4.66}-\eqref{sharp545} into \eqref{sharp1}, we get
\begin{align}\label{sharp546}
    u_p(x_{1,p})=&\frac{u_p(x_{1,p})}{ p}\left(8\pi H(x_{1,p},x_{1,p})-6\log{2}+O\left(\frac{1}{p}\right)\right)\notag\\&-\frac{4u_p(x_{1,p})\log{\varepsilon_{1,p}}}{ p}\left(1-\frac{3}{p}+O\left(\frac{1}{p^{2-\delta}}\right)\right)\notag\\&+\frac{u_p(x_{2,p})}{ p}\left(8\pi G(x_{1,p},x_{2,p})+O\left(\frac{1}{p}\right)\right)+O\left(\frac{1}{p^{2}}\right).
\end{align}

On the other hand,
note from \eqref{pup8} that
$\frac{u_p(x_{2,p})}{u_p(x_{1,p})}=-1+o(1)$, so \eqref{sharp546} yields
\begin{align}   \label{pup22} 
\frac{4\log{\varepsilon_{1,p}}}{p}=-1+O\left(\frac{1}{p}\right).
\end{align}
Inserting $\varepsilon_{1,p}^{-2}=pu_p(x_{1,p})^{p-1}$ into \eqref{pup22} leads to
\begin{align}\label{c4.70}
u_p(x_{1,p})=\sqrt{e}+O(\log{p}/p).
\end{align}
Similarly, 
\begin{align}\label{c4.701}
    u_p(x_{2,p})=-\sqrt{e}+O(\log{p}/p).
\end{align}
These together imply
\begin{align}\label{eq-4-75}
     u_p(x_{2,p})=-u_p(x_{1,p})+O(\log{p}/p).
\end{align}
Inserting $\eqref{eq-4-75}$ into $\eqref{sharp546}$ and dividing by $u_{p}(x_{1,p})$, we obtain
\begin{align}\label{pup25}
    1=&\frac{1}{ p}\left(8\pi H(x_{1,p},x_{1,p})-6\log{2}+O\left(\frac{1}{p}\right)\right)\notag\\&-\frac{4\log{\varepsilon_{1,p}}}{ p}\left(1-\frac{3}{p}+O\left(\frac{1}{p^{2-\delta}}\right)\right)\notag\\&-\frac{1}{ p}\left(8\pi G(x_{1,p},x_{2,p})+O\left(\frac{1}{p}\right)\right)+O\left(\frac{1}{p^{2-\delta}}\right).
\end{align}
Write $x^*_p=(x_{1,p},x_{2,p})$ and
\[\Psi_1(x^*_p)=G(x_{1,p},x_{2,p})-H(x_{1,p},x_{1,p})=G(x_{1,p},x_{2,p})-R(x_{1,p}),\]
then \eqref{pup25} implies
\begin{align}\label{eq-4-65}
    \log{\varepsilon_{1,p}}&=-\frac{p}{4}\left(\frac{1+\frac{1}{p}\left(8\pi\Psi_1(x^*_p)+6\log{2}\right)+O\left(\frac{1}{p^{2-\delta}}\right)}{1-\frac{3}{p}+O\left(\frac{1}{p^{2-\delta}}\right)}\right)\\
   & =-\frac{p}{4}\left(1+\frac{1}{p}\left(8\pi\Psi_1(x^*_p)+6\log{2}+3\right)+O\left(\frac{1}{p^{2-\delta}}\right)\right)\nonumber.
\end{align}
Again by inserting $\varepsilon_{1,p}^{-2}=pu_p(x_{1,p})^{p-1}$ into \eqref{eq-4-65}, we finally obtain
\[
    u_p(x_{1,p})=\sqrt{e}\left(1-\frac{\log p}{p}+\frac{1}{p}\left(4\pi\Psi_1(x^*_p)+3\log{2}+2\right)\right)+O\left(\frac{1}{p^{2-\delta}}\right).
\]
By a similar argument, we can prove
\[
     u_p(x_{2,p})=-\sqrt{e}\left(1-\frac{\log p}{p}+\frac{1}{p}\left(4\pi\Psi_2(x^*_p)+3\log{2}+2\right)\right)+O\left(\frac{1}{p^{2-\delta}}\right),
\]
where $\Psi_2(x^*_p)=G(x_{1,p},x_{2,p})-H(x_{2,p},x_{2,p})=G(x_{1,p},x_{2,p})-R(x_{2,p})$. This completes the proof.\end{proof}

\begin{remark} Recalling $C_{i,p}$ defined in \eqref{pup43}, it follows from \eqref{sharp533} and \eqref{c4.70}-\eqref{c4.701} that
\begin{align}\label{pup44}
    C_{i,p}=\frac{|u_p(x_{i,p})|}{p}\int_{B_{\frac{\varrho}{\varepsilon_{i,p}}}(0)}\left|1+\frac{v_{i,p}(z)}{p}\right|^p dz=\frac{8\pi\sqrt{e}}{p}\left(1+O\left(\frac{\log p}{p}\right)\right).
\end{align}

\end{remark}

\begin{proof}[Proof of Theorem \ref{cor-1.6}] Assume by contradiction that up to subsequence,
\begin{equation}\label{pup30}p\left|u_p(x_p^+)+u_p(x_p^-)\right|\rightarrow+\infty.\end{equation}
Since 
Theorem \ref{th2} implies
\begin{align}
    p\left|u_p(x_p^+)+u_p(x_p^-)\right|=4\pi\sqrt{e}\left|R(x_p^+)-R(x_p^-)\right|+O(1),
\end{align}
we obtain $R(x_p^+)-R(x_p^-)\to \infty$. 

On the other hand,
it follows from \eqref{pup1}-\eqref{pup8} that up to a subsequence, there exist $x^{\pm}\in\Omega$ such that $\lim_{p\to\infty}x_{p}^{\pm}=x^{\pm}$, so $R(x_p^+)-R(x_p^-)\to R(x^+)-R(x^-)\neq \infty$,  a contradiction. This completes the proof.
\end{proof}

\begin{proof}[Proof of Corollary \ref{prop sharp-1}] 
Recalling the definition \eqref{11} of $J_p$, we have
\begin{align}
    J_p(u_p)=\left(\frac{1}{2}-\frac{1}{p+1}\right)\int_{\Omega}|u_p|^{p+1} dx.
\end{align}
Similarly as before, it follows from \eqref{pup5} that
\begin{align}\label{eq-4-82}
    p\int_{\Omega}|u_p|^{p+1}=&p\sum\limits_{i=1}^2 \int_{B_{\frac{\varrho_0}{2}}(x_{i,p})}|u_p|^{p+1} dx +O\left(\frac{C^p}{p^{p-1}}\right)\notag\\
    =&\sum\limits_{i=1}^2\left|u_p(x_{i,p})\right|^2 \int_{B_{\frac{\varrho_0}{2\varepsilon_{i,p}}}(0)}\left|1+\frac{v_{i,p}}{p}\right|^{p+1} dz +O\left(\frac{C^p}{p^{p-1}}\right).
\end{align}
By Theorem $\ref{th2}$ we get
\begin{align}
    \left|u_p(x_{i,p})\right|^2=e\left(1 +\frac{2}{p}\left(-\log{p}+3\log{2}+2\right)+\frac{8\pi }{p}\Psi_i(x^*_p)+O\left(\frac{1}{p^{2-\delta}}\right)\right),
\end{align}
where $x_p^*=(x_{1,p},x_{2,p})$.
This, together with $\eqref{sharp535^*}$, shows that
\begin{align*}
    &\left|u_p(x_{i,p})\right|^2 \int_{B_{\frac{\varrho_0}{2\varepsilon_{i,p}}}(0)}\left|1+\frac{v_{i,p}}{p}\right|^{p+1} dz\\
    =&8\pi e \left( 1+\frac{1}{p}\left(-2\log{p}+6\log{2}-1+8\pi\Psi_i(x^*_p)\right)\right)+O\left(\frac{1}{p^{2-\delta}}\right).
\end{align*}
Inserting this into \eqref{eq-4-82}, it follows that
\begin{align}
    \int_{\Omega}|u_p|^{p+1}=\frac{16\pi e}{p}+\frac{16\pi e}{p^2}\left(-2\log{p}+6\log{2}-1\right)+\frac{64\pi^2e}{p^2}\Psi(x_p^*)+\frac{o(1)}{p^2},
\end{align}
which finally implies \eqref{eq-4-80} by using $\Psi(x_p^*)\to \Psi(x^*)$.
\end{proof}

\begin{proof}[Proof of Corollary \ref{th-7-2}]
Let $(u_p)_{p>1}$ be a family of least energy nodal solutions to $\eqref{p_p}$ with the concentrate points $\{x^+, x^-\}$. Assume by contradiction that $x^*=(x^+, x^-)$ is not a minimum point of $\Psi$, i.e.   there is a minimun point $\xi^*=(\xi^+,\xi^-)\in \mathcal{M}$ of $\Psi$ such that 
\begin{align}\label{eq-7-6}
    \Psi (x^*)>\Psi(\xi^*)=\min\limits_{\mathcal{M}} \Psi(x,y).
\end{align}
Since $\xi^*=(\xi^+,\xi^-)$ is a minimun point of $\Psi$, by the finite-dimensional reduction method, Esposito, Musso and Pistoia constructed the existence of low energy nodal solutions $v_p$ for $p$ large with the concentrate points $\{\xi^+, \xi^-\}$; see \cite[Theorem 1.2]{19}. Then by \eqref{eq-4-80} and  \eqref{eq-7-6}, we have $J_p(u_p)>J_p(v_p)$ for $p$ large, a contradiction with that fact that $u_p$ is a least energy nodal solution. This proves $\Psi (x^*)=\min\limits_{\mathcal{M}} \Psi(x,y)$.
\end{proof}

\section{Non-degeneracy of the low energy  nodal solutions}

In this section, we consider the non-degeneracy of low-energy solutions $(u_p)_{p>1}$ of \eqref{p_p}, and prove Theorem \ref{thm-1.6}.
The following ideas are inspired by \cite{4}.
Suppose by contradiction that there exists a sequence $\xi_p\in H_0^1(\Omega)$ such that
\begin{align}\label{linear op}
    \Vert \xi_p\Vert_{\infty}=1\quad\text{and}\quad -\Delta \xi_p =p|u_p|^{p-1}\xi_p\,\,\text{in}\,\,\Omega.
\end{align}
\begin{Lemma}\label{eq-5.2}
Define $\xi_{i,p}(z):=\xi_p(\varepsilon_{i,p}z+x_{i,p})$. Then up to a subsequence, we have as $p\rightarrow+\infty$,
\begin{align}\label{eq-5-2}
    \xi_{i,p}(z)=a_i\frac{8-|z|^2}{8+|z|^2}+\sum\limits_{j=1}^2\frac{b_{ij}z_j}{8+|z|^2}+o(1)\quad\text{in $C_{loc}^2(\mathbb{R}^2)$,}
\end{align}
 where $a_i$ and $b_{ij}$ are some constants.
\end{Lemma} 
\begin{proof}
Clearly
\begin{align}
    -\Delta\xi_{i,p}(z)=\left|1+\frac{v_{i,p}(z)}{p}\right|^{p-1}\xi_{i,p}(z)\quad\text{ in }\quad \Omega_{i,p}=\frac{\Omega-x_{i,p}}{\varepsilon_{i,p}}.
\end{align}
Remark that by Lemma \ref{prop44}, a similar proof as Corollary \ref{re52} implies
\begin{align}\label{pup32}
    \left|1+\frac{v_{i,p}(z)}{p}\right|^{p-1}=O\left(\frac{1}{1+|z|^{4-\delta}}\right),\quad \forall |z|\leq \frac{\varrho_0}{\varepsilon_{i,p}},
\end{align}
These, together with standard elliptic estimates, $v_{i,p}\to U$ and Lemma \ref{Linear repe}, imply the desired estimate $\eqref{eq-5-2}$.
\end{proof}
\begin{Lemma}\label{lemma-5.2} Fix any $\varrho\in (0, \frac{\varrho_0}{2}]$. Then
up to a subsequence,
\begin{align}\label{eq-5-3}
    \xi_p(x)=\sum\limits_{i=1}^2 \left(A_{i,p} G(x_{i,p},x)+\sum_{j=1}^2 B_{ij,p}\partial_{j}G(x_{i,p},x) \right)+o(\varepsilon_p)
\end{align}
in $C_{loc}^1(\overline{\Omega}\setminus\cup_{i=1}^2 B_{2\varrho}(x_{i,p}))$, where $\partial_{j}G(x,z)=\frac{\partial G(x,z)}{\partial x_j}$, $\varepsilon_p=\max\{\varepsilon_{1,p},\varepsilon_{2,p}\}$,
\begin{align}\label{eq-5.4}
    A_{i,p}:=p\int_{B_{\varrho}(x_{i,p})}\left|u_p(y)\right|^{p-1}\xi_p(y)dy
\end{align}
and 
\begin{align}\label{eq-5.5}
    B_{ij,p}:=p\int_{B_{\varrho}(x_{i,p})}\left(y-x_{i,p}\right)_j\left|u_p(y)\right|^{p-1}\xi_p(y)dy.
\end{align}
Moreover,
\begin{align}\label{A-eq-es}
  p A_{i,p}=8\pi a_i+o(1),    
\end{align}
\begin{align}\label{B-eq-es}
    B_{ij,p}=2\pi b_{ij}\varepsilon_{i,p}+o(\varepsilon_{i,p}).
\end{align}

\end{Lemma}

\begin{proof} Fix any $x\in \overline{\Omega}\setminus\cup_{i=1}^2 B_{2\varrho}(x_{i,p})$. By the Green representation formula and the Taylor expansion,
\begin{align}\label{pup31}
    \xi_p(x)=\sum\limits_{i=1}^2 & p\int_{B_{\varrho}(x_{i,p})}G(y,x)\left|u_p(y)\right|^{p-1}\xi_p(y)dy+O\left(\frac{C^{p-1}}{p^{p-2}}\right)\notag\\
    =\sum\limits_{i=1}^2 &\left(A_{i,p} G(x_{i,p},x)+\sum_{j=1}^2 B_{ij,p}\partial_{j}G(x_{i,p},x) \right)\notag\\
    & +O\left(\sum\limits_{i=1}^2 p\int_{B_{\varrho}(x_{i,p})}|y-x_{i,p}|^2\left|u_p(y)\right|^{p-1}dy\right) +o(\varepsilon_p).
\end{align}
Note from \eqref{pup32} that
\begin{align}\label{eq-5-8}
    &p\int_{B_{\varrho}(x_{i,p})}|y-x_{i,p}|^2\left|u_p(y)\right|^{p-1}dy\nonumber\\
=&\varepsilon_{i,p}^2\int_{B_{\frac{\varrho}{\varepsilon_{i,p}}}(0)} |z|^2\left|1+\frac{v_{i,p}(z)}{p}\right|^{p-1} dz=O\left(\varepsilon_{i,p}^2\int_{B_{\frac{\varrho}{\varepsilon_{i,p}}}(0)}\frac{|z|^2 dz}{1+|z|^{4-\delta}}\right)\nonumber\\
=&O\left(\varepsilon_{i,p}^{2-\delta}\right)=o(\varepsilon_p).  
\end{align}
Inserting \eqref{eq-5-8} into \eqref{pup31}, we obtain that \eqref{eq-5-3} holds in $C_{loc}(\overline{\Omega}\setminus\cup_{i=1}^2 B_{2\varrho}(x_{i,p}))$. A similar argument shows \eqref{eq-5-3} in $C_{loc}^1(\overline{\Omega}\setminus\cup_{i=1}^2 B_{2\varrho}(x_{i,p}))$.

Next, we estimate $B_{ij,p}$ and $A_{i,p}$.  By \eqref{eq-5.5}, \eqref{pup32}, \eqref{eq-5-2} and the dominated convergent theorem, we have,
\begin{align}\label{pup33}
    \frac{B_{ij,p}}{\varepsilon_{i,p}}=&\int_{B_{\frac{\varrho}{\varepsilon_{i,p}}}(0)} z_j \left|1+\frac{v_{i,p}}{p}\right|^{p-1}\xi_{i,p}dz\nonumber\\
    \to&\int_{\mathbb{R}^2}z_j e^{U(z)}\left(a_i\frac{8-|z|^2}{8+|z|^2}+\sum\limits_{j=1}^2\frac{b_{ij}z_j}{8+|z|^2}\right) dz=2\pi b_{ij},\;\;\text{as }p\to\infty,
\end{align}
where we used the fact that the integral of odd functions is zero. This proves \eqref{B-eq-es}. 

On the other hand, we can write
\begin{align}\label{eq-5-16}
    A_{i,p}=&\frac{p}{u_p(x_{i,p})}\int_{B_{\varrho}(x_{i,p})}\left(u_p(x_{i,p})-u_p\right)\left|u_p\right|^{p-1}\xi_p dx\notag\\ &+\frac{p}{|u_p(x_{i,p})|
    }\int_{B_{\varrho}(x_{i,p})}\left|u_p\right|^{p}\xi_p dx\notag\\
    =&-\frac{1}{p}\int_{B_{\frac{\varrho}{\varepsilon_{i,p}}}(0)}v_{i,p}\left|1+\frac{v_{i,p}}{p}\right|^{p-1}\xi_{i,p}dz +\frac{p}{|u_p(x_{i,p})|
    }\int_{B_{\varrho}(x_{i,p})}\left|u_p\right|^{p}\xi_p dx.
\end{align}
Similarly as \eqref{pup33}, we can deduce 
\begin{align}\label{eq-5-17}
    &\int_{B_{\frac{\varrho}{\varepsilon_{i,p}}}(0)}v_{i,p}\left|1+\frac{v_{i,p}}{p}\right|^{p-1}\xi_{i,p}dz\notag\\
    \to &\int_{\mathbb{R}^2}U(z) e^{U(z)}\left(a_i\frac{8-|z|^2}{8+|z|^2}+\sum\limits_{j=1}^2\frac{b_{ij}z_j}{8+|z|^2}\right) dz
    =-8\pi a_i,\;\;\text{as }p\to\infty.
\end{align}
Since \eqref{est-v-1} implies $\frac{p}{|u_p(x_{i,p})|
    }\int_{B_{\varrho}(x_{i,p})}\left|u_p\right|^{p}\xi_p dx=O(1)$, so
 $A_{i,p}=O(1)$. 
Then
by Lemma \ref{B.l1}, \eqref{eq-5-3} and \eqref{B-eq-es},
it follows that 
\begin{align*}
    &(p-1)\int_{B_{\varrho}(x_{i,p})}|u_p|^{p-1}u_p\xi_p dx\\
    =&\int_{B_{\varrho}(x_{i,p})} \left(\xi_p\Delta u_p -u_p\Delta \xi_p \right)dx=\int_{\partial B_{\varrho}(x_{i,p})} \left(\xi_p\frac{\partial u_p}{\partial \nu}  -u_p\frac{\partial \xi_p}{\partial \nu} \right)dS\\
    =&(A_{1,p}C_{2,p}+A_{2,p}C_{1,p})\times\\
    &\int_{\partial B_{\varrho}(x_{i,p})}\left(G(x_{2,p},x)\frac{\partial G(x_{1,p},x)}{\partial \nu}-G(x_{1,p},x)\frac{\partial G(x_{2,p},x)}{\partial \nu}\right)dS
     +O\left(\frac{\varepsilon_{p}}{p}\right)\\
    =&O\left(\frac{\varepsilon_{p}}{p}\right),
\end{align*}
thus
\begin{align}\label{eq-5-20}
    \int_{B_{\varrho}(x_{i,p})}|u_p|^{p}\xi_p dx=(-1)^{i+1}\int_{B_{\varrho}(x_{i,p})}|u_p|^{p-1}u_p\xi_p dx =O\left(\frac{\varepsilon_{p}}{p^2}\right).
\end{align}
From \eqref{eq-5-16}, \eqref{eq-5-17} and \eqref{eq-5-20}, we finally obtain
\begin{align}
    p A_{i,p}=8\pi a_i+o(1)+O\left(\varepsilon_{p}\right)=8\pi a_i+o(1).
\end{align}
This completes the proof.
\end{proof}

The following two quadratic forms and Pohozaev identities are inspired by \cite{4}. 
For $i=1,2$ and $j=1,2$, define
\begin{align}
    P_i(f,g):=-2\varrho&\int_{\partial B_{\varrho}(x_{i,p})}\langle \nabla f, \nu \rangle \langle \nabla g, \nu \rangle dS_y+ \varrho\int_{\partial B_{\varrho}(x_{i,p})}\langle \nabla f, \nabla g \rangle dS_y,\\
    Q_{ij}(f,g):=&\int_{\partial B_{\varrho}(x_{i,p})}\left(-\frac{\partial f}{\partial x_j}\frac{\partial g}{\partial\nu} -\frac{\partial f}{\partial\nu }\frac{\partial g}{\partial x_j} +\langle \nabla f, \nabla  g\rangle \right)dS_y.
\end{align}
where $f,g\in C^2(\Omega)$ and $\nu(x)$ denotes the outer normal at $\partial B_{\varrho}(x_{i,p})$ at $x$. 
\begin{Lemma}[Pohozaev identity]\label{lemma-5.4}
Let $u_p$ be  a solution of $\eqref{p_p}$ and $\xi_p$ is defined in $\eqref{linear op}$. For any fixed $\varrho>0$, we have for $i=1,2$ and $j=1,2$,
\begin{align}
    P_i(u_p, \xi_p)=\varrho &\int_{\partial B_{\varrho}(x_{i,p})} \left|u_p\right|^{p-1}u_p\xi_p dS_y - 2 \int_{B_{\varrho}(x_{i,p})} \left|u_p\right|^{p-1}u_p\xi_p dy,\label{eq-5-24}\\
    Q_{ij}(u_p,\xi_p)
    =&\int_{\partial B_{\varrho}(x_{i,p})} \left|u_p\right|^{p-1}u_p\xi_p\nu_j dS_y,\label{eq-5-25}
\end{align}
where $\nu(x)=(\nu_1(x),\nu_2(x))$ denotes the outer normal at $\partial B_{\varrho}(x_{i,p})$ at $x$. 
\end{Lemma}

\begin{proof} The proof is almost the same as \cite[Proposition 2.6]{4} where such type of Pohozaev identities were established for positive solutions, so we omit the details here.
\end{proof}

\begin{Lemma}\cite[Proposition 2.4]{4}\label{lemma-5.5}
For any fixed $\varrho>0$, we have for $i=1,2$, $m=1,2$, $n=1,2$, $j=1,2$ and $h=1,2$,
\begin{align}
    &P_i(G(x_{n,p},y),G(x_{m,p},y))=\begin{cases}
        -\frac{1}{2\pi},\quad &\text{  for } m=n=i,\\
        0 , \quad &\text{ for other cases },
    \end{cases}\\
    &P_i(G(x_{n,p},y),\partial_h G(x_{m,p},y))=\begin{cases}
        \frac{1}{2}\partial_h R(x_{i,p}),\quad &\text{  for } m=n=i\\
         -D_h G(x_{n,p},x_{i,p}),\quad &\text{  for } m=i, n\neq i ,\\
        0 , \quad &\text{ for other cases } ,
    \end{cases}\\
    &Q_{ij}(G(x_{n,p},y), G(x_{m,p},y))=\begin{cases}
        \partial_j R(x_{i,p}),\quad &\text{  for } m=n=i,\\
        D_j G(x_{n,p},x_{i,p}),\quad &\text{  for } m=i, n\neq i ,\\
        D_j G(x_{m,p},x_{i,p}),\quad &\text{  for } n=i, m\neq i ,\\
        0 , \quad &\text{  for other cases },
    \end{cases}\\
    &Q_{ij}(G(x_{n,p},y),\partial_h G(x_{m,p},y))=\begin{cases}
        \partial_h\partial_j R(x_{i,p}),\quad &\text{  for } m=n=i,\\
        \partial_h D_j G(x_{m,p},x_{i,p}),\quad &\text{  for } n=i, m\neq i ,\\
        D_h D_j G(x_{n,p},x_{i,p}),\quad &\text{  for } m=i, n\neq i ,\\
        0 , \quad &\text{  for other cases },
    \end{cases}
\end{align}
where $R(x):=H(x,x)$, $\partial_j G(x,y):=\frac{\partial G(x,y)}{\partial x_j}$ and $D_j G(x,y):=\frac{\partial G(x,y)}{\partial y_j}$.
\end{Lemma}

Now we will introduce the following matrix
\begin{equation}\label{wrons}
W(x,y)=\begin{pmatrix}\partial_{1}^2\Psi & \partial_{1}\partial_{2}\Psi & -D_1 \partial_1 \Psi & -D_1 \partial_2 \Psi\\ \partial_{2}\partial_{1}\Psi & \partial_{2}^2 \Psi & -D_2 \partial_1 \Psi & -D_2 \partial_2 \Psi\\
-\partial_1 D_1  \Psi &  -\partial_2 D_1  \Psi & D_{1}^2\Psi & D_{1}D_{2}\Psi\\
 -\partial_1 D_2 \Psi &  -\partial_2  D_2 \Psi & D_{2}D_{1}\Psi & D_{2}^2\Psi
\end{pmatrix}
\end{equation}
where $\partial_{j}\Psi(x,y):=\frac{\partial \Psi(x,y)}{\partial x_j}$, $D_{j}\Psi(x,y):=\frac{\partial \Psi(x,y)}{\partial y_j}$ for $j=1,2$, $x=(x_1,x_2), y=(y_1,y_2)$ and $\Psi$ is defined in $\eqref{eq-1-13^2}$.
\begin{remark}\label{re-5-6}
$\Psi$ is non-degenerate at $(x^+,x^-)$ if and only if $W(x^+,x^-)$ is non-degenerate, where we use the fact that for matrices $A,B,C\in \mathbb{R}^{2\times2}$,
\begin{align}
  \det\begin{pmatrix}
    A & -C\\
    -C & B
    \end{pmatrix}
    =
    \det\begin{pmatrix}
    A & C\\
    C & B
    \end{pmatrix}.
\end{align}

\end{remark}

\begin{Lemma}\label{Lemma-5.6}
Recall that $a_i$ and $b_{ij}$ are defined in Lemma $\ref{eq-5.2}$, $A_{i,p}$ is defined in $\eqref{eq-5.4}$ and $B_{ij,p}$ is defined in $\eqref{eq-5.5}$.
Suppose $\Psi$ defined in $\eqref{eq-1-13^2}$ is non-degenerate at $(x^+,x^-)$, then
\[a_i=b_{ij}=0,\quad A_{i,p}=o(\varepsilon_p),\quad B_{ij,p}=o(\varepsilon_p),\quad\forall i,j.\]
\end{Lemma}
\begin{proof} Let $\varrho\in (0, \frac{\varrho_0}{2}]$.
By Lemma \ref{B.l1} and Lemma \ref{lemma-5.2}, we have 
\begin{align}\label{pup46}
    P_{i}(u_p,\xi_p)&=\sum\limits_{m=1}^2\sum\limits_{n=1}^2 (-1)^{n+1} A_{m,p}C_{n,p} P_i(G(x_{n,p},y),G(x_{m,p},y))\\+\sum\limits_{m=1}^2&\sum\limits_{n=1}^2\sum\limits_{j=1}^2 (-1)^{n+1} C_{n,p}B_{mj,p} P_i(G(x_{n,p},y),\partial_jG(x_{m,p},y))+o\left(\frac{\varepsilon_p}{p}\right).\nonumber
\end{align}
We claim that
\begin{align}\label{pup40}
\sum\limits_{m=1}^2\sum\limits_{n=1}^2\sum\limits_{j=1}^2 (-1)^{n+1} C_{n,p}B_{mj,p} P_i(G(x_{n,p},y),\partial_jG(x_{m,p},y))=o\left(\frac{\varepsilon_p}{p}\right).
\end{align}
We only prove \eqref{pup40} for $i=1$ (the case $i=2$ can be proved similarly). By Lemma \ref{lemma-5.5} we have
\begin{align}
\text{LHS of \eqref{pup40}}=\sum_{j=1}^2B_{1j,p}\left(\frac{C_{1,p}}{2}\partial_jR(x_{1,p})-C_{2,p}D_jG(x_{2,p},x_{1,p})\right).
\end{align}
Since \eqref{pup44} and Theorem A-(2) imply
\begin{align*}&p\left(\frac{C_{1,p}}{2}\partial_jR(x_{1,p})-C_{2,p}D_jG(x_{2,p},x_{1,p})\right)\\
\to& 8\pi\sqrt{e}\left(\frac{1}{2}\partial_jR(x^+)-D_jG(x^-,x^+)\right)=0,\quad\text{as }p\to\infty,\end{align*}
we conclude from $B_{1j,p}=O(\varepsilon_p)$ that \eqref{pup40} with $i=1$ holds.

By  \eqref{pup44}, \eqref{A-eq-es}, \eqref{pup40} and Lemma \ref{lemma-5.5}, we see from \eqref{pup46} that
\begin{align}\label{eq-5-31}
    P_{i}(u_p,\xi_p)= \frac{(-1)^{i}}{2\pi}A_{i,p}C_{i,p}+o\left(\frac{\varepsilon_p}{p}\right)=\frac{(-1)^{i}32\pi\sqrt{e}a_i}{p^2}+o\left(\frac{1}{p^2}\right).
\end{align}
Meanwhile, by \eqref{pup5}, \eqref{eq-5-20} and \eqref{eq-5-24}, we have 
\begin{align}\label{eq-5-32}
    P_{i}(u_p,\xi_p)=O\left(\frac{C^{p}}{p^{p}}\right)+O\left(\frac{\varepsilon_p}{p^2}\right)=o\left(\frac{1}{p^2}\right).
\end{align}
Hence, \eqref{eq-5-31}-\eqref{eq-5-32} imply $a_i=0$ and $A_{i,p}=o(\varepsilon_{p})$.

To prove $b_{ij}=0$, we see from Lemma \ref{B.l1} and Lemma \ref{lemma-5.2} that (note $A_{m,p}=o(\varepsilon_p)$ and $C_{n,p}=O(1/p)$)
\begin{align}\label{eq-5-33}
    Q_{ij}(u_p,\xi_p)=&\sum\limits_{m=1}^2\sum\limits_{n=1}^2 (-1)^{n+1} A_{m,p}C_{n,p} Q_{ij}(G(x_{n,p},y),G(x_{m,p},y))\notag\\
    +\sum\limits_{h=1}^2&\sum\limits_{m=1}^2\sum\limits_{n=1}^2 (-1)^{n+1} B_{mh,p}C_{n,p} Q_{ij}(G(x_{n,p},y),\partial_h G(x_{m,p},y))+o\left(\frac{\varepsilon_p}{p}\right)\notag\\
    =\sum\limits_{h=1}^2&\sum\limits_{m=1}^2\sum\limits_{n=1}^2 (-1)^{n+1} B_{mh,p}C_{n,p} Q_{ij}(G(x_{n,p},y),\partial_h G(x_{m,p},y))+o\left(\frac{\varepsilon_p}{p}\right).
\end{align}
Meanwhile,  \eqref{pup5} and \eqref{eq-5-25} together imply 
\begin{align}\label{eq-5-36}
    Q_{ij}(u_p,\xi_p)=O\left(\frac{C^{p}}{p^{p}}\right)=o\left(\frac{\varepsilon_p}{p}\right).
\end{align}
Hence, from \eqref{eq-5-33}-\eqref{eq-5-36} and Lemma \ref{lemma-5.2}, we obtain
\begin{align}
    \sum\limits_{h=1}^2&\left((-1)^{i+1} B_{ih,p}C_{i,p}\partial_h\partial_jR(x_{i,p})
    +(-1)^{i^*+1} B_{ih,p}C_{i^*,p}D_h D_j G(x_{i^*,p},x_{i,p})\right.\notag\\
    &+\left.(-1)^{i+1} B_{i^*h,p}C_{i,p}\partial_h D_j G(x_{i^*,p},x_{i,p})\right)=o\left(\frac{\varepsilon_p}{p}\right), 
\end{align}
where $\{i,i^*\}=\{1,2\}$. Hence,
\begin{align}\label{pup47}
    \sum\limits_{h=1}^2&\left(\frac{B_{ih,p}}{\varepsilon_p}\partial_h\partial_j\left(R(x_{i,p})-2\frac{C_{i^*,p}}{C_{i,p}}G(x_{i,p},x_{i^*,p})\right)\right.\notag\\
    &+\left.2\frac{B_{i^*h,p}}{\varepsilon_p}\partial_j D_h G(x_{i,p},x_{i^*,p})\right)=o(1),\quad\forall i,j.
\end{align}
Since $\lim_{p\to\infty}\frac{C_{i^*,p}}{C_{i,p}}=1$ and $\Psi$ is non-degenerate at $(x^+,x^-)$, i.e. the matrix $W(x^+,x^-)$ is non-degenerate, it follows from \eqref{pup47} that
$B_{ij,p}=o(\varepsilon_p)$ for any $i,j\in \{1,2\}$. 
Since \eqref{eq-4-65} says that \begin{align}\label{pup50}|\log \varepsilon_{i,p}|=\frac{p}{4}+O(1),\end{align} so $\frac1C\leq |\frac{\varepsilon_{1,p}}{\varepsilon_{2,p}}|\leq C$ for some $C>1$, and then we conclude from
\eqref{B-eq-es} that
\begin{align*}
   b_{ij}=\frac{\varepsilon_p}{\varepsilon_{i,p}}o(1)=o(1),
\end{align*}
i.e. $b_{ij}=0$ and so $B_{ij,p}=o(\varepsilon_p)$ for all $i,j$.
\end{proof}

\begin{corollary}\label{cor-5.7}
Fix any $\varrho\in (0, \frac{\varrho_0}{2}]$ and $R>0$, we have
\begin{align}
    \Vert\xi_{p}\Vert_{L^{\infty}(\overline{\Omega}\setminus\cup_{i=1}^2 B_{2\varrho}(x_{i,p}))}=o(\varepsilon_p)\quad\text{ and }\quad \Vert\xi_{p}\Vert_{L^{\infty}(\cup_{i=1}^2B_{\varepsilon_{i,p}R}(x_{i,p}))}=o(1).
\end{align}
\end{corollary}
\begin{proof}
This is a direct consequence of Lemmas \ref{eq-5.2}-\ref{lemma-5.2} and \ref{Lemma-5.6}.
\end{proof}

\begin{Lemma}\label{Lemma5-8}
Fix any $\varrho\in (0, \frac{\varrho_0}{2}]$, we have
\begin{align}\label{pup56}
    \Vert\xi_{p}\Vert_{L^{\infty}\left(\cup_{i=1}^2
     B_{2\varrho}(x_{i,p})\setminus B_{2p \varepsilon_{i,p}}(x_{i,p})\right) }=o(1).
\end{align}
\end{Lemma}

\begin{proof} Fix any $x\in B_{2\varrho}(x_{i,p})\setminus B_{2p \varepsilon_{i,p}}(x_{i,p})$.
Similarly as \eqref{pup31},
\begin{align}\label{eq-5-46}
    \xi_p(x)
    =\sum\limits_{i=1}^2 &p\int_{B_{p\varepsilon_{i,p}}(x_{i,p})}G(y,x)\left|u_p(y)\right|^{p-1}\xi_p(y)dy\notag\\
    &+\sum\limits_{i=1}^2 p\int_{B_{\varrho}(x_{i,p})\setminus B_{p\varepsilon_{i,p}}(x_{i,p})}G(y,x)\left|u_p(y)\right|^{p-1}\xi_p(y)dy+o(\varepsilon_p).
\end{align}
Since $G(y,x)=O(|\log|y-x||)$ for $y\in B_{\varrho}(x_{i,p})$,
it follows from \eqref{pup32} that
\begin{align*}
    &p\int_{B_{\varrho}(x_{i,p})\setminus B_{p\varepsilon_{i,p}}(x_{i,p})}G(y,x)\left|u_p(y)\right|^{p-1}\xi_p(y)dy\\
    =&O\left(p\int_{B_{\varrho}(x_{i,p})\setminus B_{p\varepsilon_{i,p}}(x_{i,p})}|\log|y-x||\left|u_p(y)\right|^{p-1}dy\right)\\
    =&O\left(\int_{B_{\frac{\varrho}{\varepsilon_{i,p}}}(0)\setminus B_{p}(0)}|\log|\varepsilon_{i,p}z+x_{i,p}-x||\left|1+\frac{v_{i,p}(z)}{p}\right|^{p-1}dy\right)\\
    =&O\left(\int_{B_{\frac{\varrho}{\varepsilon_{i,p}}}(0)\setminus B_{p}(0)}\frac{|\log \varepsilon_{i,p}|}{1+|z|^{4-\delta}}dz\right)+
    O\left(\int_{B_{\frac{\varrho}{\varepsilon_{i,p}}}(0)\setminus B_{p}(0)}\frac{\left|\log\left|z-\frac{x-x_{i,p}}{\varepsilon_{i,p}}\right|\right|}{1+|z|^{4-\delta}}dz\right).
\end{align*}
By \eqref{pup50} we have
\begin{align}\label{eq-5-48}
    \int_{B_{\frac{\varrho}{\varepsilon_{i,p}}}(0)\setminus B_{p}(0)}\frac{|\log \varepsilon_{i,p}|}{1+|z|^{4-\delta}}dz=O\left(\frac{1}{p^{1-\delta}}\right).
\end{align}
Furthermore, since $2p\varepsilon_{i,p}\leq |x-x_{i,p}|\leq 2\varrho$, it was proved in \cite[(4.40)]{4} that
\begin{align}\label{eq-5-51}
     \int_{B_{\frac{\varrho}{\varepsilon_{i,p}}}(0)\setminus B_{p}(0)}\frac{\left|\log\left|z-\frac{x-x_{i,p}}{\varepsilon_{i,p}}\right|\right|}{1+|z|^{4-\delta}}dz=O\left(\frac{1}{p^{1-\delta}}\right).
\end{align}
Thus
\begin{align}\label{eq-5-47}
p\int_{B_{\varrho}(x_{i,p})\setminus B_{p\varepsilon_{i,p}}(x_{i,p})}G(y,x)\left|u_p(y)\right|^{p-1}\xi_p(y)dy=O\left(\frac{1}{p^{1-\delta}}\right).
\end{align}
On the other hand, we have 
\begin{align}\label{eq-5-52}
    &p\int_{B_{p\varepsilon_{i,p}}(x_{i,p})}G(y,x)\left|u_p(y)\right|^{p-1}\xi_p(y)dy\notag\\
    =&G(x_{i,p},x)p\int_{B_{p\varepsilon_{i,p}}(x_{i,p})}\left|u_p(y)\right|^{p-1}\xi_p(y)dy\notag\\
    &+p\int_{B_{p\varepsilon_{i,p}}(x_{i,p})}\left(G(y,x)-G(x_{i,p},x)\right)\left|u_p(y)\right|^{p-1}\xi_p(y)dy.
\end{align}
Since $2p\varepsilon_{i,p}\leq |x-x_{i,p}|\leq 2\varrho\ll 1$, by \eqref{eq-5.4}, $A_{i,p}=o(\varepsilon_p)$ and \eqref{pup32}, we have
\begin{align}\label{eq-5-53}
   &G(x_{i,p},x)p\int_{B_{p\varepsilon_{i,p}}(x_{i,p})}\left|u_p(y)\right|^{p-1}\xi_p(y)dy\notag\\
    =&G(x_{i,p},x)\left(A_{i,p}-\int_{B_{\frac{\varrho}{\varepsilon_{i,p}}}(0)\setminus B_p(0)}\left|1+\frac{v_{i,p}(z)}{p}\right|^{p-1}\xi_{i,p}(z)dz\right)\notag\\
    =&O\left(\log|p\varepsilon_{i,p}|\right)\left(o(\varepsilon_p)+O\left(\frac{1}{p^{2-\delta}}\right)\right)=o(1).
\end{align}
Since $\nabla_y G(y,x)=O(\frac{1}{|y-x|})$, we have for $y\in B_{p\varepsilon_{i,p}}(x_{i,p})$ that
\begin{align*}
      G(y,x)-G(x_{i,p},x)=&\nabla_y G(\upsilon_p y+(1-\upsilon_p)x_{i.p},x)(y-x_{i,p})\\
   =&O\left(\frac{|y-x_{i,p}|}{|(\upsilon_p y+(1-\upsilon_p)x_{i.p})-x|}\right)=O\left(\frac{|y-x_{i,p}|}{p\varepsilon_{i,p}}\right),
\end{align*}
where $\upsilon_p \in (0,1)$ and we used $|(\upsilon_p y+(1-\upsilon_p)x_{i.p})-x|\geq |x-x_{i,p}|-|y-x_{i,p}|\geq p\varepsilon_{i,p}$ in the last equality. It follows that
\begin{align}\label{eq-5-54}
    &p\int_{B_{p\varepsilon_{i,p}}(x_{i,p})}\left(G(y,x)-G(x_{i,p},x)\right)\left|u_p(y)\right|^{p-1}\xi_p(y)dy\notag\\
    =&O\left(\frac{1}{p\varepsilon_{i,p}}\right)O\left(\varepsilon_{i,p}\int_{ B_p(0)}|z|\left|1+\frac{v_{i,p}(z)}{p}\right|^{p-1}dz\right)=O\left(\frac{1}{p}\right).
\end{align}
Thus, from \eqref{eq-5-46} and  \eqref{eq-5-47}-\eqref{eq-5-54}, we obtain the desired estimate \eqref{pup56}.
\end{proof}

Now we are in the position to prove Theorem \ref{thm-1.6}.

\begin{proof}[Proof of Theorem \ref{thm-1.6}] Let $x_p$ be the maximum point of $|\xi_p|$ in $\Omega$, and without loss of generality, we may assume $\xi_p(x_p)>0$, i.e. 
\[\xi_p(x_p)=1=\|\xi_p\|_{L^\infty}.\]
Then by Corollary \ref{cor-5.7} and Lemma \ref{Lemma5-8}, we have 
\begin{align}
    x_p\in B_{2p \varepsilon_{i,p}}(x_{i,p})\setminus B_{ \varepsilon_{i,p}R}(x_{i,p})\quad\text{for some $i\in\{1,2\}$}.
\end{align}
Denote $r_p=|x_p-x_{i,p}|$ such that $R \varepsilon_{i,p}\leq r_p\leq 2p \varepsilon_{i,p}$. Define 
\begin{align}\label{pup66}
    \xi_{i,p}^*(z):=\xi_p(r_p z+x_{i,p})=\xi_{i,p}\left(\frac{r_p }{\varepsilon_{i,p}}z\right),
\end{align}
which satisfies 
\begin{align}
    -\Delta \xi_{i,p}^*(z)=\left(\frac{r_p}{\varepsilon_{i,p}}\right)^2\left|1+\frac{v_{i,p}(\frac{r_p }{\varepsilon_{i,p}}z)}{p}\right|^{p-1}\xi_{i,p}^*(z).
\end{align}
Since $R \varepsilon_{i,p}\leq r_p\leq 2p \varepsilon_{i,p}$, by \eqref{pup32} we have
\begin{align}
    \left(\frac{r_p}{\varepsilon_{i,p}}\right)^2\left|1+\frac{v_{i,p}(\frac{r_p }{\varepsilon_{i,p}}z)}{p}\right|^{p-1}=O\left(\frac{\varepsilon_{i,p}^{2-\delta} }{r_p^{2-\delta}\left|z\right|^{4-\delta}}\right),\quad\forall |z|\leq \frac{\varrho_0}{r_p}.
\end{align}
Since $R$ can be take arbitarily large,
by standard elliptic estimates, up to a subsequence we have $\xi_{i,p}^* \rightarrow \xi_0^*$ in $C_{loc}^2(\mathbb{R}^2\setminus\{0\})$ such that 
$\Delta \xi_0^*=0$ in $\mathbb{R}^2\setminus\{0\}$. Since $|\xi_0^*|\leq 1$, the singularity $0$ is removable. By $\xi_{i,p}^*\left(\frac{x_p-x_{i,p}}{r_p}\right)=1$, it follows that $\xi_0^*\equiv 1$, namely
\begin{equation} \label{pup67}
\xi_{i,p}^* \rightarrow 1\quad\text{in }C_{loc}(\mathbb{R}^2\setminus\{0\}).
\end{equation}

On the other hand, recall Lemma $\ref{eq-5.2}$ that $\xi_{i,p}$ satisfies that
\begin{align}
    -\Delta\xi_{i,p}-e^U\xi_{i,p}=f_p:=\left(\left|1+\frac{v_{i,p}}{p}\right|^{p-1}-e^U\right)\xi_{i,p}.
\end{align}
By $v_{i,p}\to U$ in $C_{loc}^2(\mathbb{R}^2)$ and Lemma \ref{prop44}, the same argument as \cite[p.184]{4} shows
\begin{align}\label{pup88}
f_p(z)=O\left(\frac{1}{p(1+|z|^{4-\delta})}\right),\quad\forall |z|\leq \frac{\varrho_0}{\varepsilon_{i,p}}.
\end{align}

Write $\xi_{i,p}(z)=\xi_{i,p}(|z|,\theta)$ for $z=|z|e^{i\theta}$, we define
\begin{align}
    \widetilde{\xi}_{i,p}(|z|):=\frac1{2\pi}\int_0^{2\pi}\xi_{i,p}(|z|,\theta) d\theta,
\end{align}
which sloves 
\begin{align}
     -\left(\widetilde{\xi}_{i,p}\right)''-\frac{1}{r}\left(\widetilde{\xi}_{i,p}\right)'-e^U\widetilde{\xi}_{i,p}=\widetilde{f}_p:=\frac1{2\pi}\int_0^{2\pi}f_{p}(|z|,\theta) d\theta.
\end{align}
By standard ODE's theory, there exist $c_{1,p},c_{2,p}\in\mathbb{R}$ such that
\begin{align}
    \widetilde{\xi}_{i,p}=c_{1,p}\widetilde{\psi}_1+c_{2,p}\widetilde{\psi}_2+V,
\end{align}
where 
\begin{align}\label{vrr0}
    \widetilde{\psi}_1(r)=\frac{8-r^2}{8+r^2}\quad\text{ and }\quad\widetilde{\psi}_2(r)=\frac{(8-r^2)\log r+16}{8+r^2}
    \end{align}
    are two linearly independent solutions of $-\psi''-\frac1r \psi'-e^U\psi=0$, and
    \begin{align}\label{vrr}
    V(r)=\widetilde{\psi}_1(r)\int_0^r s \widetilde{\psi}_2(s) \widetilde{f}_p(s)ds-\widetilde{\psi}_2(r)\int_0^r s \widetilde{\psi}_1(s) \widetilde{f}_p(s)ds.
\end{align}
Since $\widetilde{\xi}_{i,p}$ is bounded and $V(0)=0$, we have $c_{2,p}=0$. Then by Corollary $\ref{cor-5.7}$, we get 
\begin{align*}
    o(1)=\xi_{i,p}(0)= \widetilde{\xi}_{i,p}(0)=c_{1,p}\widetilde{\psi}_1(0)=c_{1,p},
\end{align*}
i.e. $c_{1,p}=o(1)$. Note from \eqref{pup88} that
\begin{align*}
    \widetilde{f}_p=O\left(\frac{1}{p(1+|z|^{4-\delta})}\right),\quad \forall |z|\leq \frac{\varrho_0}{\varepsilon_{i,p}},
\end{align*}
which together \eqref{vrr0}-\eqref{vrr} shows that 
\begin{align}\label{eq-5-66}
    \left|\widetilde{\xi}_{i,p}\left(\frac{r_p}{\varepsilon_{i,p}}\right)\right|=\left|V\left(\frac{r_p}{\varepsilon_{i,p}}\right)\right|+o(1)=O\left(\frac{\log{\frac{r_p}{\varepsilon_{i,p}}}}{p}\right)+o(1)=o(1).
\end{align}
However, by \eqref{pup66} and \eqref{pup67} we have
\begin{align*}
    \widetilde{\xi}_{i,p}\left(\frac{r_p}{\varepsilon_{i,p}}\right)=\frac1{2\pi}\int_0^{2\pi}\xi_{i,p}\left(\frac{r_p}{\varepsilon_{i,p}},\theta\right) d\theta=\frac1{2\pi}\int_0^{2\pi}\xi_{i,p}^*\left(1,\theta\right) d\theta= 1+o(1),
\end{align*}
which is a contradiction with $\eqref{eq-5-66}$. The proof is complete.
\end{proof}

\begin{remark}
Here is an example that if $(x^+,x^-)$ is a degenerate critical point of $\Psi$, then $u_p$ may be degenerate. Consider the case that $\Omega=B_1(0)=\{x : |x|<1\}$ is the unit ball and let $u_p$ be a least energy nodal solution of \eqref{p_p}. Then it was proved in \cite{21} that $u_p$ is foliated Schwarz symmetric and in particular, not radially symmetric. Since any rotation $u_p(Ax)$ is still a nodal solution of \eqref{p_p} (where $A \in SO(2)$), we conclude that $u_p(x)$ is degenerate for any $p$.

On the other hand, let $x^+, x^-$ are the concentrate points of $u_p$ as stated in Theorem A. Since we have proved 
that the condition $p(\|u_p^+\|_{\infty}-\|u_p^-\|_{\infty})=O(1)$ holds, it follows from \cite[Theorem 5]{14} that
there exists $\theta_0\in [0,2\pi)$ such that 
\begin{align}
    x^+=(r^*\cos{\theta_0}, r^*\sin{\theta_0}) \quad \text{and} \quad x^-=(-r^*\cos{\theta_0}, -r^*\sin{\theta_0}),
\end{align}
where $r^*= \sqrt{-2+\sqrt{5}}$. 
Again by taking rotations, we can assume $\theta_0=0$, i.e. the concentrate points are
\[
   x^+=\left(\sqrt{-2+\sqrt{5}},0\right) \quad \text{and} \quad x^-=\left(-\sqrt{-2+\sqrt{5}},0\right).
\]
Since the Green function $G(x,y)$ on $B_1(0)$ is known, by a direct calculation, we get
\[
   W(x^+,x^-)=\frac{1}{8\pi}\begin{pmatrix}-13-5\sqrt{5} & 0 & -7-3\sqrt{5} & 0\\ 0 & -1-\sqrt{5} & 0 & 1+\sqrt{5} \\
   -7-3\sqrt{5} & 0 & -13-5\sqrt{5} & 0\\
 0 & 1+\sqrt{5} & 0 & -1-\sqrt{5}
\end{pmatrix},
\]
and so $\det W(x^+,x^-)=0$.
This, together with Remark \ref{re-5-6}, shows that $(x^+,x^-)$ is a degenerate critial point of $\Psi$.
\end{remark}

\subsection*{Acknowledgements} The research of Z. Chen was supported by NSFC (No. 12222109, 12071240).


\begin{thebibliography}{10}



\bibitem{asy1-3}
{\sc Adimurthi and M. Grossi},
{\it Asymptotic estimates for a two-dimensional problem with polynomial nonlinearity.}
 Proc. Amer. Math. Soc. {\bf 132}(2004), 1013-1019.

\bibitem{23}
{\sc T.~Bartsch, A.~M. Micheletti, and A.~Pistoia}, {\em The {M}orse property
  for functions of {K}irchhoff-{R}outh path type}, Discrete Contin. Dyn. Syst.
  Ser. S, 12 (2019), pp.~1867--1877.



\bibitem{11}
{\sc C.-C. Chen and C.-S. Lin}, {\em Sharp estimates for solutions of
  multi-bubbles in compact {R}iemann surfaces}, Comm. Pure Appl. Math., 55
  (2002), pp.~728--771.

\bibitem{10}
{\sc W.~Chen and C.~Li}, {\em Methods on nonlinear elliptic equations}, vol.~4
  of AIMS Series on Differential Equations \& Dynamical Systems, American
  Institute of Mathematical Sciences (AIMS), Springfield, MO, 2010.

  
\bibitem{CCZ} {\sc Z.~{Chen}, Z.~{Cheng}, and H.~{Zhao}}, {\em {Asymptotic behavior of least energy nodal solutions for biharmonic Lane-Emden problems in dimension four}}, arXiv e-prints,  (2023), p.~arXiv:arXiv:2306.04416.



\bibitem{15}
{\sc F.~De~Marchis, M.~Grossi, I.~Ianni, and F.~Pacella}, {\em
  {$L^\infty$}-norm and energy quantization for the planar {L}ane-{E}mden
  problem with large exponent}, Arch. Math. (Basel), 111 (2018), pp.~421--429.

\bibitem{1}
{\sc F.~De~Marchis, I.~Ianni, and F.~Pacella}, {\em Asymptotic analysis and sign-changing bubble towers for Lane--emden problems}, Journal of the
  European Mathematical Society, 17 (2015), pp.~2037--2068.

\bibitem{8}
\leavevmode\vrule height 2pt depth -1.6pt width 23pt, {\em Asymptotic profile
  of positive solutions of Lane-Emden problems in dimension two}, J. Fixed
  Point Theory Appl., 19 (2017), pp.~889--916.

\bibitem{13}
\leavevmode\vrule height 2pt depth -1.6pt width 23pt, {\em Asymptotic analysis
  for the {L}ane-{E}mden problem in dimension two}, in Partial differential
  equations arising from physics and geometry, vol.~450 of London Math. Soc.
  Lecture Note Ser., Cambridge Univ. Press, Cambridge, 2019, pp.~215--252.

\bibitem{12}
{\sc P.~Esposito, M.~Musso, and A.~Pistoia}, {\em Concentrating solutions for a
  planar elliptic problem involving nonlinearities with large exponent},
  Journal of Differential Equations, 227 (2006), pp.~29--68.

\bibitem{19}
{\sc P.~Esposito, M.~Musso, and A.~Pistoia}, {\em On the existence and profile
  of nodal solutions for a two-dimensional elliptic problem with large exponent
  in nonlinearity}, Proc. Lond. Math. Soc. (3), 94 (2007), pp.~497--519.

\bibitem{2}
{\sc M.~Grossi, C.~Grumiau, and F.~Pacella}, {\em Lane-emden problems:
  Asymptotic behavior of low energy nodal solutions}, Annales de L'Institut
  Henri Poincare Section (C) Non Linear Analysis, 30 (2013), pp.~121--140.

\bibitem{14}
\leavevmode\vrule height 2pt depth -1.6pt width 23pt, {\em Lane {E}mden
  problems with large exponents and singular {L}iouville equations}, J. Math.
  Pures Appl. (9), 101 (2014), pp.~735--754.

\bibitem{4}
{\sc M.~Grossi, I.~Ianni, P.~Luo, and S.~Yan}, {\em Non-degeneracy and local
  uniqueness of positive solutions to the lane-emden problem in dimension two},
  Journal de Math{\'e}matiques Pures et Appliqu{\'e}es, 157 (2022),
  pp.~145--210.



\bibitem{21}
{\sc F.~Pacella and T.~Weth}, {\em Symmetry of solutions to semilinear elliptic
  equations via {M}orse index}, Proc. Amer. Math. Soc., 135 (2007),
  pp.~1753--1762.





\bibitem{7}
{\sc X.~Ren and J.~Wei}, {\em On a two-dimensional elliptic problem with large
  exponent in nonlinearity}, Trans. Amer. Math. Soc., 343 (1994), pp.~749--763.

\bibitem{wei-jmaa}
\leavevmode\vrule height 2pt depth -1.6pt width 23pt, {\em On a semilinear
  elliptic equation in {${\bf R}^2$} when the exponent approaches infinity}, J.
  Math. Anal. Appl., 189 (1995), pp.~179--193.


\end{thebibliography}
\end{document}